\documentclass[12pt]{amsart}
\usepackage{hyperref}
\usepackage{xcolor}
\usepackage{amsfonts}
\usepackage{amssymb}
\usepackage{amsxtra}
\usepackage{amstext}
\usepackage[english]{babel}
\usepackage{enumerate}
\usepackage{mathtools}
\usepackage[normalem]{ulem}
\usepackage{mathrsfs} 
\usepackage{latexsym}
\usepackage{nicefrac}
 

\def\1{{\mathchoice {1\mskip-4mu\mathrm l}      
{1\mskip-4mu\mathrm l} 
{1\mskip-4.5mu\mathrm l} {1\mskip-5mu\mathrm l}}}

\newtheorem{theorem}{Theorem}[section]
\newtheorem{lemma}[theorem]{Lemma}

\newtheorem{proposition}[theorem]{Proposition}

\newtheorem{remark}[theorem]{Remark}

\newtheorem{definition}{Definition}

\newtheorem*{assumption}{Assumption W}

\def\P{{\mathbb P}} 
\def\E{{\mathbb E}} 
\def\Z{{\mathbb Z}} 

\def\N{{\mathbb{N}}}
\def \R {{\mathbb{R}}}

 \newcommand{\ssup}[1] {{{\scriptscriptstyle{({#1}})}}} 
\def\comment#1{} 

\newcommand{\Fcal}   {{\mathcal F }}

\newcommand{\Pcal}   {{\mathcal P }}

\newcommand{\X}{{\bf X}}
\newcommand{\Y}{{\bf Y}}

\def\ignore#1{}

\def\ti{\to\infty}

\newcommand{\rmd}{\,\mathrm{d}}
\def\phi{\varphi}

\def\limt{\lim_{t\to\infty}}
\def\limt0{\lim_{t\to 0}}

\def\|{\,|\,}

\title{Vertex-reinforced jump process on the integers with nonlinear reinforcement}
\author{Andrea Collevecchio}
\address[\sc A. Collevecchio]{School of Mathematics, Monash University, Victoria 3800, Australia}
\email{andrea.collevecchio@monash.edu}

\author{Tuan-Minh Nguyen}
\address[\sc T.M. Nguyen]{School of Mathematics, Monash University, Victoria 3800, Australia}
\email{tuanminh.nguyen@monash.edu}

\author{Stanislav Volkov}
\address[\sc S. Volkov]{Centre for Mathematical Sciences, Lund University, Lund 22100-118, Sweden}
\email{stanislav.volkov@matstat.lu.se}

\begin{document}
\begin{abstract}
We consider a non-linear vertex-reinforced jump process (VRJP($w$)) on $\mathbb{Z}$  with  an increasing measurable weight function $w:[1,\infty)\to [1,\infty)$ and  initial weights equal to one. Our main goal is to study the asymptotic behaviour of VRJP($w$) depending on the integrability of  the reciprocal of $w$. In particular, we prove that if {$\int_1^{\infty} \frac{\rmd u}{w(u)} =\infty$} then the process is recurrent, i.e. it visits each vertex infinitely often and all local times are unbounded. On the other hand, if  {$\int_1^{\infty} \frac{\rmd u}{w(u)} <\infty$} and there exists a $\rho>0$ such that 
$t \mapsto w(t)^{\rho}\int_t^{\infty}\frac{du}{w(u)}$ is non-increasing then the process will eventually get stuck on exactly three vertices, and there is only one vertex with unbounded local time. We also show that if the initial weights are all the same, VRJP on $\mathbb{Z}$ cannot be transient, i.e. there exists at least one vertex that is visited infinitely often. Our results extend the ones previously obtained by Davis and Volkov [Probab. Theory Relat. Fields (2002)] who showed that VRJP with linear reinforcement on $\mathbb{Z}$ is recurrent.  
\end{abstract}
\subjclass[2010]{60G17, 60K35, 60G20}
\keywords{self-interacting processes, random processes with reinforcement, vertex-reinforced jump processes, localization}
\maketitle

\tableofcontents

\section{Introduction}
Reinforced random walks attracted the attention of many researchers in the past 30 years. 
This field  started with the seminal work of Coppersmith and Diaconis \cite{CD}. They introduced linear edge-reinforced  random walk (LERRW) which can be roughly described as follows. It is a discrete time process. It takes values on the vertices of a locally finite graph, and at each step it jumps to nearest neighbour vertices with transition probabilities that are updated each time, according to the following rule. {To each edge, we assign} an initial positive weight, and each time the edge is traversed its weight is increased by one. The probability that the process traverses at time $t+1$ a given edge $e$,  incident to the position of the walk at time $t$, is proportional to the total weight of $e$ by time $t$. This process was extensively studied (see, e.g. \cite{ KR1999, MOR2018}). In particular, a long-standing open problem in the field was the study of LERRW on $\Z^d$, with $d \ge 2$, and, in particular,  to establish whether  this process is recurrent/transient, depending on its initial weights. {The existence of a recurrent phase was proved} by Sabot and Tarr\`es \cite{ST1} and Angel, Crawford and Kozma \cite{ACK}, using different approaches.
For an alternative proof, see \cite{CZ19}.  In particular, \cite{ST1} created a link between LERRW and a super--symmetric hyperbolic sigma model called $H^{2|2}$.   This connection was achieved by first proving that LERRW is a time change of a continuous time reinforced process, called Vertex-Reinforced Jump Process (VRJP). They then showed that the latter process is a mixture of markovian jump processes, and the mixing measure is exactly the partition function from $H^{2|2}$ model. {Using the relation between VRJP and LERRW, the existence of a transient phase of LERRW in dimension $d \ge 3$ was showed by Disertori, Sabot and Tarr\`es \cite{D2015}; the transience in dimension $d=2$ for any initial constant weights was proved by Sabot and Zeng \cite{SZ2019}; and the uniqueness of the phase transition was finally proved by Poudevigne \cite{Poudevigne}.}

VRJP is the main object studied in the present paper. VRJP was conceived by W. Werner and first studied in \cite{DV2002}. It is a continuous time process, which reinforces the vertices instead of the edges, and it jumps to nearest neighbours with probability proportional to their local times.
Given the past and the actual position of the particle, the time of the next jump is  exponentially distributed  with average equal to the sum of the local times of the neighbours of vertex currently visited. Roughly speaking, the larger the weight, the more likely is the vertex to be visited, creating a ``rich-gets-richer" effect. 
A formal definition of VRJP is given in Section \ref{se:main}.
For a more detailed literature review on this proces see Section~\ref{LitVRJP}. 
As this process reinforces the vertices, its behaviour should be compared to the one of the so-called vertex-reinforced random walk (VRRW). The latter is a discrete-time process, which jumps to the nearest neighbours, and  each time a vertex is visited, its weight is increased by one.  The probability that it chooses to jump to a particular neighbour $x$ is proportional to {the weight of $x$ at the time} of the jump.  
This process, when defined on $\Z$, localizes on a finite set \cite{PV1999} and was proved to localize on exactly {\bf five} vertices by Tarr\`es \cite{T2004}.
In a sequence of papers, it has been considered non-linear versions of VRRW where the probability to jump to a given vertex $x$ is proportional to a weight function $w$ applied to the weight of $x$ by the time of the jump.
Notably, up to this time {(see \cite{BSS14a}, \cite{BSS14}, \cite{T2004} and \cite{V06}) VRRW on $\Z$} have been proved to localize on $2, 4, 5, 7, 9, \cdots$ vertices depending on the strength of their reinforcement weight function $w$. {In particular, the authors in \cite{BSS14a, BSS14} showed that VRRW cannot localize on exactly 3 points and they also conjectured that} there exists a $w$ such that the corresponding VRRW is recurrent but spends  asymptotically all of its time on only three sites.
On the other hand, {strongly {edge}-reinforced} random walk localizes on one edge. This was proved on $\Z$ by Davis \cite{Dav} and on general graph and general reinforcement function $w$, satisfying $\sum_i 1/w(i) <\infty$, by Cotar and Thacker \cite{Tha}. See also \cite{Aka, LT2007, LT2008}. 

In this paper we study a general version of VRJP on $\Z$ with initial weights equal to 1, where the probability to jump from a vertex to another is determined by an increasing function $w$ of the local times. We call $w \colon [1, \infty) \mapsto [1, \infty)$ the weight  (or reinforcement) function. Our main results can be summarized as follows (See Theorem~\ref{th:main} below for a precise statement). If  $\int_1^{\infty}\frac{\rmd u}{w(u)}<\infty$ and there exists a $\rho>0$ such that $t\mapsto w(t)^{\rho}\int_t^{\infty}\frac{du}{w(u)}$ is non-increasing on $[1,\infty)$ then VRJP localizes on {\bf exactly three} vertices. Moreover, all the local times are a.s. finite with the exception of one vertex.  On the other hand, if $\int_1^{\infty}\frac{\rmd u}{w(u)}=\infty$ then VRJP is recurrent, i.e. it visits each vertex infinitely often, and all the local times are unbounded.

\section{Model and main result}\label{se:main}
In this section, we define rigorously a generalized version of VRJP, on general graphs. 
\noindent 
Let $G=(V,E)$ be a locally finite connected, undirected graph without loops, i.e. without edges whose endpoints coincide, where $V$ and $E$  stand, respectively, for the set of vertices and the set of edges. Fix a collection of non-negative real numbers $\ell = (\ell_v)_{v \in V}$ and a measurable weight function $w:[0,\infty)\to [0,\infty)$. We consider a continuous time process ${\bf X} = (X_t)_t$, so-called general {vertex-reinforced jump process} with parameters $(\ell, w)$, denoted VRJP($\ell, w$), which takes values on the vertices of $G$,  and is defined as follows. 
\begin{enumerate}
\item[i.] It is a c\`adl\`ag process, and jumps between nearest neighbour vertices.
\item[ii.] For each time $t>0$, conditionally to the past  $\mathcal{F}_t := \sigma\{X_{s},s\le t\}$, the probability that  exactly one jump occurs  during  the time interval $(t,t+h]$, and is  towards a neighbour $v$ of $X_t$  is given by
\begin{equation} \label{eq:smint}
w\left(\ell_v+\int_0^t \1_{\{X_s=v\}}\text{d}s \right)\cdot h+o(h),
\end{equation}
The probability of more than one jump  in $(t, t+h]$ is $o(h)$.
\end{enumerate} 
For each vertex $v\in V$, we set  
$$L(v,t)=\ell_v+\int_0^t\1_{\{X_s=v\}}\text{d}s,$$
i.e.  the total time spent at vertex $v$ up to time $t$ plus the initial value $\ell_v$. We call $(L(v,t), v\in V)_{t\ge0}$ the \textbf{local times process} of $\X$. {Set $c=\inf_{v\in V} \ell_v$}. The process $\X$ is either called \textbf{strongly reinforced} if {$\int_{c}^{\infty}\frac{\rmd u}{w(u)}<\infty$} or \textbf{weakly reinforced} if ${\int_{c}^{\infty}\frac{\rmd u}{w(u)}=\infty}$.   When $\ell_v \equiv 1$ for all $v\in V$,  we simply call the process defined above VRJP($w$). 

This paper focuses on VRJP($w$) defined on $\Z$. For simplicity, unless we state otherwise, we assume that this process starts from $0$ and with an abuse of notation, we identify $\Z$ with the set of its vertices. Additionally, {we require that the weight function} $w$ is (strictly) increasing. Notice that the law of $\mathbf{X}$ will not change if the weight function $w$ is replaced by its right-continuous modification (which always exists for any increasing function). To summarize, the following assumption will be needed throughout the paper. 
\begin{assumption}
\label{assu:W}  We assume that $w\colon [0, \infty) \mapsto [0, \infty)$ is measurable and increasing. {For simplicity, we assume additionally} that $w$ is right-continuous, $w(0)=0$ and $w(1) =1$.
\end{assumption}
We can now formulate our main result as follows.

\begin{theorem}\label{th:main}  Let $\X$ be VRJP($w$) on $\Z$, where $w$ satisfies Assumption-W.  
\begin{enumerate}
\item[(a)]  If $\X$ is weakly reinforced, i.e.
$$ \int_{1}^\infty \frac {{\rm d}s}{w(s)}   =\infty,$$
then $\X$ is recurrent i.e. it visits every point infinitely often and $L(v,\infty)=\infty$ a.s. for all $v\in \mathbb{Z}$.
\item[(b)] If $\X$ is strongly reinforced, i.e.
$$ \int_{1}^\infty \frac {{\rm d}s}{w(s)}   <\infty,$$
and there exists a $\rho>0$ such that the map
\begin{equation}\label{eq:assu} t \mapsto w(t)^{\rho}\int_t^{\infty}\frac{\rmd s}{w(s)}\quad \text{is non-increasing}\end{equation} 
then the $\X$ localizes on exactly 3 vertices and exhibits unbounded local time in exactly one vertex. More precisely, a.s. {there exists a unique} vertex $v \in \Z$ such that  $L(v, \infty) = \infty$,  whereas $L(x, \infty) < \infty$, for all $x  \neq v$, and the set $\Z \setminus\{v-1, v, v+1\}$ is visited finitely many times.
\end{enumerate}
\end{theorem}

\begin{remark}
{1. The condition that $w$ is strictly increasing is essential for the validity of Proposition~\ref{ineq.L}. This proposition is later applied to the proof of the non-transience of $\X$ (see Theorem \ref{th:notransient} below).}

{2. Our result can be easily extended  to the case when $\ell_x=c$ for all $x\in \Z$ and $w(c)>0$,  for some $c>0$.}

3. We believe, and conjecture  that Theorem \ref{th:main} (b) holds without assuming \eqref{eq:assu}. On the other hand,  {we can replace \eqref{eq:assu} by simpler criteria if $w$ has some additional regular property. For example, assuming that $w$ is differentiable in $[1, \infty)$, \eqref{eq:assu} is fullfiled if} 
$$\sup_{t\ge 1} w'(t)\int_{t}^{\infty}\frac{\rmd s}{w(s)}<\infty.$$
{Furthermore, assuming that $w\in C^2([1,\infty))$, it is easy to check by L'H\^opital's rule {that \eqref{eq:assu} holds} true if} $$\limsup_{t\to\infty}\frac{w'(t)^2}{w(t)w''(t)}<\infty.$$ 
{The class of functions $w$ that satisfy \eqref{eq:assu} is therefore quite large. In particular, it contains any power function $w(x) = x^a$ with $a>1$, and exponential functions $w(x) = {\rm e}^{a x}$ with $a>0$.}
\end{remark}

Even though we were not able to drop condition  \eqref{eq:assu} to prove localization, we were able to prove that $\X$ cannot be transient as long  as  $w$ satisfies Assumption-W.

\begin{theorem}\label{th:notransient} Suppose $w$ satisfies Assumption-W. VRJP($w$) on $\Z$ cannot be transient, that is
$$\P(\lim_{t \ti} |X_t| = \infty) =0. $$
\end{theorem}

\section{Literature review and Outline of proof}\label{LitVRJP}
Linear vertex-reinforced jump process seems to play a major role in the class of reinforced processes, and has been extensively studied in recent years   \cite{DV2002, DV2004, C2006, C2009, BHS, OM2018, SZ2019, Poudevigne}. As mentioned before, linear vertex-reinforced jump process (VRJP) {can be seen as variation of} linearly edge-reinforced random walk introduced by Coppersmith and Diaconis. Moreover, VRJP  is intimately connected to  super-symmetric hyperbolic sigma model  \cite{ST1}),  also called $H^{2|2}$ model, which was introduced by Zirnbauer \cite{Zi1} to model the conductance of electrons.  Under a suitable time change, VRJP is a mixture of markovian jump processes. In fact, this is the  only nearest neighbour jump process with local dependence on the occupation times satisfying the partially exchangeable property \cite{Ze1}.

 Our methods in the present paper differ considerably from the previous ones used to study VRRW, ERRW and linear VRJP. For example, the method introduced in \cite{Tha} for ERRW does not seem to adapt  well to estimate continuous local times of VRJP. The stochastic approximation approach used in \cite{OM2018} (see also \cite{P1992}, \cite{B1997} for VRRW) is only suitable the model on complete graph with polynomial reinforcement. On the other hand, previous proofs for recurrence and transience of VRJP with linear reinforcement (see e.g. \cite{DV2002}, \cite{DV2004}, \cite{SZ2019}) are heavily based on the fact the model is a mixture of markovian jump processes and thus its limiting distribution (after a time scaling) is computable. However, VRJP with general reinforcement (i.e. the weight function is not necessary linear) is, in general, NOT a mixture of markovian jump processes. This makes it hard to find an exact density formula for the local times. {On the other hand, the relation between LERRW and linear VRJP relies on a Kendall's result for the Yule process (see \cite{ST1}). We  were not able to establish a similar relation between nonlinear VRJP and nonlinear ERRW. In fact, it seems quite difficult  to extend Kendall's result to birth processes with nonlinear intensities 
 (see \cite{AN}, p.~130 for a further discussion).}


Our paper studies the case where the reinforcement is not necessarily linear and our approach can be summarised as follows.\\
 $\bullet$  We show that when all the vertices of $\Z$ have the same initial weight, VRJP($w$) cannot be transient, as long as the reinforcement function $w$ is increasing. Our proof  is quite technical but relies on a few simple observations. The ``technicalities" arise when we extrapolate information about the behaviour of the whole process from its ``local properties".  In particular, for general reinforcement we cannot apply at all the martingale approach used  in \cite{DV2002} to prove recurrence for the linear case.
First, it is enough to focus on VRJP($w$) on the non-negative integers. Loosely speaking, if the process were transient, the total local time at $0$ would have the same distribution as the total local time at 1. We show that if they were finite (which happens in the case of transience), we would have contradiction, as $0$ has an ``advantage" over $1$. In the proof, we introduce a family of coupled VRJP on connected subsets of $\Z$ and use the so-called Restriction Principle (see \cite{DV2002, DV2004, C2009} or Section~\ref{Se:No-Transience} below) to relate the behaviour of the process on $\Z^+$ to VRJP defined on subsets. \\
 $\bullet$ In order to prove the recurrence of VRJP in weak reinforcement regime, we first investigate VRJP$(w)$ $(\widetilde{X}_t)_{t\ge0}$ on the graph with two connected vertices labeled by $0$ and $1$. In particular, we consider the c\`adl\`ag finite variation process $\mathbf{Z}=(Z_t)_{t\ge0}$ given by
 $$Z_t=\int_1^{{\widetilde L}(0,t)}\frac{\rmd u}{w(u)}-\int_1^{\widetilde L(1,t)}\frac{\rmd u}{w(u)}-\frac{\1_{\{\widetilde X_t=1\}}}{w(\widetilde L(0,t))w(\widetilde L(1,t))},$$
 where $\widetilde L(0,t), \widetilde L(1,t)$ are respectively the local times at 0 and 1 up to time $t$. {Notice that if $\widetilde L(0,\infty)\wedge \widetilde L(1,\infty)<\infty$ then $Z_{t}$ tends to either $+\infty$ or $-\infty$ as $t\to\infty$. Using a martingale decomposition for $\mathbf{Z}$ (Proposition~\ref{pr:deco}), we prove that a.s.  either $Z_{t}$ is uniformly bounded or it fluctuates from $-\infty$ to $+\infty$. This allows us to prove Proposition~\ref{VRJP2vert} stating that the local times $\widetilde L(0,t)$ and $\widetilde L(1,t)$ are both unbounded as time $t$ goes to infinity.} Combining this result with the fact of non-transience and the restriction principle, we conclude part (a) of Theorem~\ref{th:main}, i.e the recurrence of weakly VRJP.\\
 $\bullet$  We  outline the proof of localization  at strong reinforcement regime, i.e. part (b) of Theorem~\ref{th:main}. We first prove a non-convergence theorem (Theorem~\ref{nonconvergence}) which is a general result applicable to a wide class of c\`adla\`g finite variation processes, and we believe is of independent interest. Applying this non-convergence result to the process $\mathbf{Z}$ defined as above, it allows us to prove Proposition~\ref{lem:liminf} asserting that under the strong reinforcement regime, $w(\widetilde{L}(0,t))$ and $w(\widetilde{L}(1,t))$ cannot grow with the same order as $t\to \infty$. {Notice that under strong reinforcement regime, $Z_{\infty}=0$ if and only if $\widetilde{L}(0,\infty)\wedge \widetilde{L}(1,\infty)=\infty$. Using Proposition \ref{lem:liminf}, we next prove Proposition~\ref{bounded} stating that 
 $Z_t$ converges a.s. to a finite non-atomic random variable as $t\to\infty$ and thus $Z_{\infty}=0$ with probability 0. As a result, the random variable $\widetilde{L}(0,\infty)\wedge \widetilde{L}(1,\infty)$ is a.s. finite and it has a non-atomic distribution. Using the restriction principle and the result on 2 vertices, we are able to study strongly VRJP on the graph with 3 consecutive connected vertices labeled by 0, 1 and 2; in particular, we rule out the case when the local times at $0$ and $2$ are both unbounded. The fact that VRJP on $\Z$ is not transient also implies that there exists one vertex with unbounded local time. Combining the above-mentioned results and the restriction principle, we thus be able to conclude that under the assumption of Theorem~\ref{th:main}(b),  VRJP$(w)$ on $\Z$ eventually localizes on exactly three points, and only one of them accumulates unbounded local time.}
 
 Our localization result should be compare to the one obtained in \cite{OM2018} where VRJP($w$) on a complete graph is studied, and $w(x) = x^\rho$, with $\rho>1$.
It was proved in \cite{OM2018} that there exists exactly one (random) vertex $v^*$ whose local time is unbounded, and all the vertices are visited infinitely often. Moreover after a random time, VRJP performs jumps only from or to $v^*$.

Our method differs substantially from the ones used in \cite{OM2018}, where a stochastic approximation approach is used. We could not implement that approach in our settings, and at same time, with our methods we are able to cover a larger class of reinforcement.

\section{Proof of Theorem~\ref{th:notransient}: non-transience of VRJP on $\mathbb{Z}$}
\label{Se:No-Transience}

In this section we  prove that if $w$ satisfies Assumption-W, then with probability 1, VRJP($w$) $\X$ on $\Z$ cannot be transient, i.e. $\X$ visits each vertex in $\Z$ finitely often. Our proofs rely on a specific construction of VRJP, which we call canonical process, which in turn allows to build a family of coupled VRJP on subsets of $\Z$.

\subsection{Canonical process(es)} We introduce a family of canonical VRJPs generated using {\bf the same} sequence of i.i.d. exponentials.    Fix a collection of non-negative initial local times $\ell=(\ell_x)_{x\in\Z}$ and a measurable weight function $w:[0,\infty)\to [0,\infty)$. To any ordered pair of consecutive integers, say $(i, j)\in \Z^2$, where $|j-i| =1$, attach a Poisson process $\Pcal(i,j)$ with rate one. We assume that the Poisson processes are independent. Denote by $(\chi^{\ssup{i,j}}_n)_{n \in \N}$ the inter-arrival times for the process  $\Pcal(i,j)$, i.e. $(\chi^{\ssup{i,j}}_n)_{n \in \N}$ are i.i.d. exponentials with mean one. We first construct the \lq skeleton\rq\ of the VRJP($\ell,w$) on $\Z$, i.e. a discrete time process which describes the  jumps of the VRJP. Let  $\tau_0 = 0$. {On the event $\{X_0 =v\}$ with $v \in \Z$, set
$$
\tau_1 = \min\left\{\frac 1{w(\ell_{v+1})} \chi^{\ssup {v,v+1}}_1, \frac 1{w(\ell_{v-1})}\chi^{\ssup {v,v-1}}_1\right\},
$$ 
and $L(v, \tau_1) = \ell_v +\tau_1$, and for $x \neq v$, let $L(x, \tau_1) = \ell_x$.} Moreover set 
$$
{X_{\tau_1} = \arg\min_{j \in \{v-1, v+1\}} \frac 1{w(\ell_j)} \chi^{\ssup {v,j}}_1.}
$$
In this definition we convey  that $a/0 = \infty$ when $a>0$. Roughly speaking, this means that all vertices $x$ with initial local time $\ell_x =0$  cannot be visited. Suppose we defined $(\tau_j, X_{\tau_j},  (L(x, \tau_j))_{x \in \Z})$ for all $j\le n$. On the event $\{X_{\tau_n} = i\}$,  {for $j\in \{i-1,i+1\}$ let $$\gamma_j =1+\sum_{k=1}^n \1_{\{(X_{\tau_{k-1}},X_{\tau_k})= (i,j)\}}$$} and  let 
 \begin{eqnarray*}
 \tau_{n+1} &=& \tau_n+ \min\left\{\frac 1{w(L(i-1, \tau_n))}\chi^{\ssup {i,i-1}}_{{\gamma_{i-1}}}, \frac 1{w(L(i+1, \tau_n))}\chi^{\ssup {i,i+1}}_{{\gamma_{i+1}}}\right\},\\
 L(i, \tau_{n+1}) &=& L(i, \tau_n) +  \tau_{n+1} - \tau_n, \quad \mbox{and for $x \neq i$ we set }  L(x,  \tau_{n+1}) = L(x, \tau_n),\\
 X_{\tau_{n+1}} &=&  \arg\min_{j \in \{i-1, i + 1\}}\frac 1{w(L(j, \tau_n))}\chi^{\ssup {i,j}}_{{\gamma_j}}.
 \end{eqnarray*} 
\begin{proposition}[Non-explosion]\label{non-explosion} Assume that {Assumption W} is fulfilled, $\ell_x>0$ for all $x\in\Z$ and
$$ {\sum_{x=1}^{\infty}\frac{1}{w(\ell_x)}=\sum_{x=1 }^{\infty}\frac{1}{w(\ell_{-x})}=\infty}.$$ 
Then VRJP($\ell,w$) $\X$ has finitely many jumps on any bounded time interval, i.e. 
 $\lim_{n \ti} \tau_n = \infty$, {almost surely.}
 \end{proposition}
 \begin{proof}
We reason by contradiction. Let $A = \{ \lim_{n \ti} \tau_n < \infty\}$ and assume that $\P(A) > 0$. We first prove that on $A$ we must have that $L(x, \infty)<\infty$ for all $x \in \Z$.  For simplicity, define $ \tau_\infty := \lim_{n \ti} \tau_n$.  By the definition of $(L(x, \tau_n))_{n \in \N}$ we have
\begin{equation}
 \label{eq:int1}
  L(x, \tau_\infty) < \ell_x+ \tau_\infty<\infty  \qquad \mbox{on $A$}, \qquad \forall x \in \Z.
  \end{equation}
{Assume w.l.o.g. that $X_0=0$}. On the event $\{\lim_{n \ti} X_{\tau_n} = \infty\}$, we have that
{$$ 
\tau_\infty \ge \sum_{x=0}^\infty \frac{1}{w(\ell_{x+1})} \chi^{\ssup{x, x+1}}_{1}.
$$
Note that $(\chi^{\ssup{x, x+1}}_{1})_{x\in \Z^+}$ is a sequence of independent exponential random variables with mean 1.}

If $\inf_{x\in \N} w(\ell_x)=0$, it immediately follows from the above inequality that $\tau_{\infty}=\infty$. On the other hand, if $w^*=\inf_{x\in \N} w(\ell_x)>0$, we have
\begin{align*}
\sum_{x=0}^\infty \frac{1}{w(\ell_{x+1})} \E\left(\chi^{\ssup{x, x+1}}_{1}\1_{\left\{\chi^{\ssup{x, x+1}}_{1} \le   w(\ell_{x+1})\right\}}\right)
&=\sum_{x=0}^\infty \frac{1}{w(\ell_{x+1})} \left(1-(w(\ell_{x+1})+1)e^{-w(\ell_{x+1})}\right)\\
& \ge K \sum_{x=0}^\infty \frac{1}{w(\ell_{x+1})}=\infty,
\end{align*}
where $K:=\inf_{u\in[w^*,\infty)}\left(1-(u+1)e^{-u}\right)>0$.
By applying Kolmogorov's Three-Series Theorem, we must have that a.s. $\sum_{x\in \Z^+}\frac{1}{w(\ell_{x+1})} \chi^{\ssup{x, x+1}}_{1} =\infty.$
Hence, on the event 
$\{\lim_{n \ti} X_{\tau_n} = \infty\}$ we have $\tau_{\infty}=\infty$ a.s., which contradicts the assumption  on $A$. We thus conclude that  the  intersection $\{\lim_{n \ti} X_{\tau_n} = \infty\} \cap A$ has measure zero.   A similar reasoning (or using symmetry) yields $\{\lim_{n \ti} X_{\tau_n} = -\infty\} \cap A$ has measure zero.

Hence on the event $A$, there exists a vertex $x \in \Z$ such that $x$ is visited infinitely often by $(X_{\tau_n})_{n \in \Z^+}.$ Notice that at least one of the neighbours of $x$ must also be visited infinitely often. Call this neighbour $y \in \{x-1, x+1\}$. Using monotonicity of $w$ and the first inequality in \eqref{eq:int1} applied to $L(y, \tau_\infty)$, we have that
$$
L(x, \tau_\infty) \ge \frac 1{w(\ell_y + \tau_{\infty})} \sum_{k=1}^\infty \chi^{\ssup{x,y}}_k = \infty, \qquad \mbox{a.s., on $A$}
$$
which contradicts~\eqref{eq:int1}.  
\end{proof}

Define the continuous time process $\X = (X_t)_{t \in \R^+}$  as follows
$$
X_t = \sum_{n = 0}^\infty X_{\tau_n} \1_{\{\tau_n \le t < \tau_{n+1}\}}.
$$
Given the properties of the exponentials, we have immediately the following.
\begin{proposition}  {For each fixed $(\ell, w)$ and fixed distribution of the starting point, the canonical process $\X$  is a VRJP($\ell, w$) on $\Z$.}
\end{proposition} 

{Denote by $\Theta$ the collection of all canonical VRJPs generated using the same sequence of i.i.d. exponentials $\big(\chi_{n}^{(i,j)}: (i,j)\in\Z^2, |i-j|=1, n\in \N\big)$.} 

\subsection{Extensions}\label{subsec:ex}
The canonical process $\X$ defined above allows us to generate {a family of coupled processes}, each indexed by connected subsets of $\Z$, giving local information about $\X$. We call these processes extensions.
\begin{definition}[Extensions]
{Let $\X$ be a canonical $VRJP(\ell,w)$ starting from 0.} Consider a connected  non-empty subset  $B$ of the graph $\Z$. 
Set 
$$\ell^{\ssup B}_x = \begin{cases}
&\ell_x \qquad \mbox{if $x \in B$}\\
&0 \qquad \mbox{ if $x \notin B$}.
\end{cases}
$$
Let $\ell^{\ssup B} = (\ell^{\ssup B}_x)_{x \in \Z}$.  {We construct $VRJP(\ell^{\ssup B} , w)$  $\X^{\ssup B}$ defined on $B$, {called the \textbf{extension} of} $\X$ on $B$, using the family of canonical processes $\Theta$. In fact, $\X^{\ssup B}$ is a canonical version with initial local times $(\ell^{\ssup B}_x)_x$ and its starting point $X_0^{\ssup B}$  chosen depending on $B$.  We distinguish three cases, a) $0 \in B$, b) $B \subseteq \N=\Z^+\setminus\{0\}$, and c) $B \subseteq -\N$.\\
\noindent{\bf  Case a)}  $X_0^{\ssup B} =0$.\\
\noindent{\bf  Case b)}  Let $b_- = \min\{i \colon i \in B\}$. We have $X_0^{\ssup B} =b_-$.\\
\noindent{\bf  Case c)} Let  $b_+ = \max\{i \colon i \in B\}$. We have $X_0^{\ssup B} =b_+$.}
\end{definition}

\begin{remark}
{Let $B$ be a nontrivial connected subset of $\Z$. For simplicity, we assume that $\ell_x>0$ for all $x\in B$. We assume additionally that $w$ satisfies Assumption $W$ and either $B$ is finite or $\sum_{x\in B}1/w(\ell_x)=\infty$. By using the similar argument as in the proof of Proposition~\ref{non-explosion}, we have that the extension $\X^{(B)}$ is also non-explosive. }
\end{remark}
 
Denote by $\left\{ L^{(B)}(i,t), i\in B,t\in[0,\infty) \right\}$ the process of local times of $\X^{(B)}$.  Set \begin{equation*}
\delta^B(u)=\inf\left\{t\ge 0: \int_0^t\1_{\{X_s\in B \}}\rmd s=u \right\} \quad \text{and}\quad T_B=\sup\{t: \delta^B(t)<\infty\}.
\end{equation*}
The process $(X_{\delta^B(t)})_{t\in [0,\infty)}$ is called the \textbf{restriction} of $\X$ on $B$. By the \textbf{restriction principle} (see e.g. \cite{DV2002}, \cite{DV2004} and \cite{C2009}), the restriction and the extension of $\X$ on $B$ must coincide up to the last exit time from $B$, i.e.  $$X_{\delta^B(t)}=X^{(B)}_t \quad \text{a.s. for all}\quad 0\le t\le T_B.$$
{As a result, we have
$$L(i,\infty)\le L^{(B)}(i,\infty) \quad \text{a.s. for all}\quad i\in B.$$
In particular, a.s. on the event $\{T_B=\infty\}$ (i.e. $\X$ spends an unbounded amount of time in $B$),  $$L(i,\infty)=L^{(B)}(i,\infty)\ \quad  \text{for all}\quad i\in B.$$
On the event $\{T_B<\infty\}$, there exists a vertex $i\in B$ such that $L(i,\infty)<L^{(B)}(i,\infty)$.}
\vskip 2em

{Now, let $\X$ be a canonical VRJP$(\ell,w)$ starting from $0$ such that $\ell_0$ and $\ell_1$ are positive. Let $\widetilde{\X}=(\widetilde{X}_t)_{t\in\R^+}$ be the extension of $\X$ on $\{0, 1\}$}. Denote by $\big(\widetilde{L}(0,t), \widetilde{L}(1,t)\big)_{t \in [0, \infty)}$ and  $(\widetilde{\tau}_j)_{j \in \Z^+}$ the process of local times of $\widetilde{\X}$ and its jump times respectively. 

{Let $\mathbf X^*=(X^*_t)_{t\in\R^+}$ be the extension of another canonical VRJP$(\ell^*,w)$ on $\{0, 1\}$ starting from 0 such that $\ell_0^*= \ell_0$ and $\ell_1^*\ge \ell_1$.} Let $(L^*(0,t),L^*(1,t))_{t \in [0, \infty)}$ and $\{\tau^*_{n}\}_{n \in \Z^+}$ stand respectively for the process of local times and the jump times of ${\X}^*$. 

 \begin{proposition}\label{comparison} Under Assumption W, we have that
  \begin{enumerate}
 \item[(i)] $\widetilde{L}(0,\widetilde{\tau}_n) \ge L^*(0,\tau^*_n)$ and   $ \widetilde{L}(1,\widetilde{\tau}_n) \le L^*(1,\tau^*_n)$ {for} all $n \in \Z^+$. Moreover, {the strict inequalities occur} when $\ell_1^*>\ell_1.$
 \item[(ii)] {Suppose that $\ell_i \equiv 1$ for all $i\in \mathbb{Z}$; $\ell_1^*=1+A$ where $A$ is an exponential random variable independent of  the sequences $(\chi^{(0,1)}_n)_{n\in \N}$ and $(\chi^{(1,0)}_n)_{n\in \N}$ with $\E[A] = 1$. Then $\big(L^*(1, \tau_{n}^*), L^*(0, \tau_{n}^*) \big)  \sim \big ( \widetilde{L}(0, \widetilde{\tau}_{n+1}), \widetilde{L}(1, \widetilde{\tau}_{n+1})\big)$ for all $n\in \Z^+.$ In particular, the process
$(L^*(0,t), (L^*(1,t))_{t\in [0, \infty)}$ is distributed as  
$(\widetilde{L}(1, t), \widetilde{L}(0, t))_{t \in [\widetilde{\tau}_{1}, \infty)}$.}
\end{enumerate} 
  \end{proposition}
\begin{proof}
In what follows, it is useful to keep in mind that $L^*(0, \tau_{2n-1}^*) = L^*(0, \tau_{2n}^*)$, $L^*(1, \tau_{2n}^*) = L^*(1, \tau_{2n+1}^*)$ and $\widetilde{L}(0, \widetilde{\tau}_{2n-1}) = \widetilde{L}(0, \widetilde{\tau}_{2n})$, $\widetilde{L}(1, \widetilde{\tau}_{2n}) = \widetilde{L}(1, \widetilde{\tau}_{2n+1})$ for all $n\in \N$.

To prove (i), it is sufficient to show that for all $k\in\Z^+$,
\begin{equation}\label{ineqL} 
L^*(1,\tau^*_{2k}){\ \ge \ }\widetilde{L}(1,\widetilde{\tau}_{2k})
\quad \text{and} \quad 
L^*(0,\tau^*_{2k+1}){\ \le \ } \widetilde{L}(0,\widetilde{\tau}_{2k+1}).
\end{equation}
For $k=0$, we have that
\begin{align*}& L^*(1, \tau_0^*)   =\ell_1^* {\ \ge \ }\ell_1 = \widetilde{L}(1,  \widetilde\tau_0),\\
 &L^*(0, \tau_{1}^*)  = \ell_0+\frac 1{w(\ell_1^*)} \chi^{\ssup{0,1}}_1 {\ \le \ } \ell_0 + \frac{1}{w(\ell_1)}\chi^{\ssup{0,1}}_1 = \widetilde{L}(0, \widetilde{\tau}_1).
\end{align*}
Suppose \eqref{ineqL} holds for all $k \in \{0,1,\dots,n-1\}$. We have that
\begin{align*}
   L^*(1, \tau_{2n }^*)   =  L^*(1, \tau_{2n-2}^*) + \frac 1{w(L^*(0, {\tau}_{2n-1}^*))} \chi^{\ssup{1,0}}_{n}
  {\ \ge \ } \widetilde{L}(1, \widetilde{\tau}_{2n-2}) +  \frac 1{w(\widetilde{L}(0, \widetilde{\tau}_{2n-1}))} \chi^{\ssup{1,0}}_{n}
  = \widetilde{L}(1, \widetilde{\tau}_{2n}),\\
  L^*(0, \tau_{2n +1}^*)     =  L^*(0, \tau_{2n-1}^*) + \frac 1{w(L^*(1, \tau_{2n}^*))} \chi^{\ssup{0,1}}_{n+1}{\ \le \ } \widetilde{L}(0, \widetilde{\tau}_{2n-1}) +  \frac 1{w(\widetilde{L}(1, \tau_{2n}))} \chi^{\ssup{0,1}}_{n+1}
 = \widetilde{L}(0, \widetilde{\tau}_{2n+1}).
\end{align*}
where the inequalities are a direct consequence of the induction rule. {It is also clear that the strict inequalities occur when $\ell_1^*>\ell_1.$} This  finishes the proof of (i). 

Next, we prove (ii). For simplicity, set $L^*_k=\big(L^*(1, \tau_{k}^*), L^*(0, \tau_{k}^*)\big)$ and $\widetilde{L}_k=\big(\widetilde{L}(0, \widetilde{\tau}_{k}), \widetilde{L}(1, \widetilde{\tau}_{k})\big)$, for all $k\in \Z^+$. 
We prove first by induction that,
 \begin{eqnarray} \label{idenL1}L^*_{2n-1}=\big( L^*(1, \tau_{2n-2}^*),L^*(0, \tau_{2n-1}^*)\big)  \sim \big ( \widetilde{L}(0, \widetilde{\tau}_{2n-1}), \widetilde{L}(1, \widetilde{\tau}_{2n})\big)=\widetilde{L}_{2n},\quad \forall n\in\N,\\
 \label{idenL2}
L^*_{2n}= \big( L^*(1, \tau_{2n}^*),L^*(0, \tau_{2n-1}^*)\big)  \sim \big ( \widetilde{L}(0, \widetilde{\tau}_{2n+1}), \widetilde{L}(1, \widetilde{\tau}_{2n})\big)=\widetilde{L}_{2n+1}, \quad \forall n\in \mathbb{Z}^+.
  \end{eqnarray}
We have
\begin{eqnarray*}
& L^*_0=(1+A,1)\sim \left(1+\chi^{\ssup{0,1}}_1,1 \right),\\
& L^*_1 = \left(1+A  ,1+\frac 1{w(1+A)} \chi^{\ssup{0,1}}_1 \right) \sim\left( 1+\chi^{\ssup{0,1}}_1,   1 +\frac 1{w(1 +\chi^{\ssup{0,1}}_1)} \chi^{\ssup{1,0}}_1 \right)= \widetilde{L}_2.
\end{eqnarray*}
 Suppose that \eqref{idenL1}-\eqref{idenL2} holds for all $m\in\{1,2,\dots,n\}$. We prove next that it {holds} for $n+1$. For any $2 \times 2$  matrix $M$,  we define $\text{diag}(M)$ to be the two-dimensional vector with the diagonal elements of $M$.
Moreover
 $$
 \begin{aligned}
L^*_{2n+1}&= \text{diag}\left(\left[ {\begin{array}{cc}
   L^*(1, \tau_{2n-2}^*) & \Big(w(L^*(0, \tau_{2n-1}^*))\Big)^{-1} \\
  L^*(0, \tau_{2n-1}^*) & \Big(w(L^*(1, \tau_{2n}^*))\Big)^{-1} \\
  \end{array} } \right]\left[ {\begin{array}{cc}
   1 & 1 \\
   \chi^{\ssup{1,0}}_{n}  & \ \chi^{\ssup{0,1}}_{n+1} \\
  \end{array} } \right]\right)\\
  &\sim \text{diag}\left(\left[ {\begin{array}{cc}
   \widetilde{L}(0, \widetilde{\tau}_{2n-1}) & \Big(w( \widetilde{L}(1, \widetilde{\tau}_{2n}))\Big)^{-1}\\
   \widetilde{L}(1, \widetilde{\tau}_{2n}) & \Big(w (\widetilde{L}(0, \widetilde{\tau}_{2n+1}))\Big)^{-1}\\
  \end{array} } \right]\left[ {\begin{array}{cc}
   1 & 1 \\
   \chi^{\ssup{0,1}}_{n+1}  & \ \chi^{\ssup{1,0}}_{n+1} \\
  \end{array} } \right]\right)\\
  &=\widetilde{L}_{2n+2}.
  \end{aligned}
 $$
  Similarly, we also have that $L^*_{2n+2}\sim \widetilde{L}_{2n+3}$.
 \end{proof}

 \subsection{Proof of non-transience}
 
It is enough to prove that VRJP($w$) defined on $\Z^+$, which is the set of non-negative integers, cannot be transient. Hence, throughout the remaining part of this section, $\X$ refers to a canonical VRJP($w$) on $\Z^+$. {We always assume that $X_0=0$, $\ell_x=1$ for all $x\in \Z^+$ and $w$ satisfies Assumption W.}

Denote by $\left(L(i, t)\colon i\in \Z^+ , t\in [0, \infty) \right)$ the process of local times related to $\X$. 

The next result rules out the following trivial case: if there is a positive probability for $\X$ to localize on a finite region, then this process cannot be transient.
\begin{proposition}\label{prop:no1}
Suppose that there exists $j \in \Z^+$ such that
\begin{equation}\label{eq:linf}
\P\left( L(j,\infty)=\infty, L(i,\infty)<  \infty, \forall i\ge j+1 \right)>0
\end{equation}
Then $\P(\X\text{ is transient})=0$.
\end{proposition}
\begin{proof}
{For $i<j$, define
\begin{equation}\label{Pij}\mathcal{P}_{[i,j]} = \sigma\Big(\chi^{\ssup {k, s}}_n\colon k,s \in\{i, i+1, \ldots,j-1,j\}, |k -s| =1, n \in \N\Big).\end{equation}}
Let \begin{equation}\label{Pi}{\mathcal{P}_{i}} = \sigma\left(\bigcup_{j=i+1}^{\infty}  \mathcal{P}_{[i,j]}\right).\end{equation}
Recall that $(\chi^{(i, i+1)}_n)_n$ and $(\chi^{(i+1, i)}_n)_n$ are independent. 
Hence, it is evident that $\mathcal{P}_{[0,i]}$ and ${\mathcal{P}_{i}}$ are independent. Consider the extension $\X^{\ssup i}$ on $i+\Z^+$ and let $L^{\ssup i}(k, \infty)$  be its local times.  Set
$$A_i = \Big\{L^{\ssup i}(i, \infty)=\infty,\ L^{\ssup i}(j, \infty)<\infty \mbox{ for all } j \ge i+1\Big\}. $$
Notice that each $A_i$ is measurable with respect to ${\mathcal{P}_{i}}$, hence $\1_{A_i}$ is a function of the exponential random variables
$$ \Xi_i :=\left \{\chi^{(j,k)}_n \mbox{ where } |j-k|=1, \mbox{ and } j\wedge k :=\min\{j,k\} \ge i, n\in\N \right\}.$$
Each $\Xi_i$ is an infinite vector with i.i.d. coordinates, distributed as exponential(1). Hence $(\Xi_i)_i$ is ergodic. Moreover, as the VRJPs $\X^{\ssup i}$, after a relabelling of the vertices, share the same distribution, we have that $A_i$s have the same probabilty, and there exists a measurable function $g$ such that 
\begin{equation}\label{eq:bir1}
\1_{A_i} = g(\Xi_i).
\end{equation}
We first prove that if \eqref{eq:linf} holds then 
\begin{equation}\label{eq:linf1}
\P\left(A_0\right)>0.
\end{equation}
For all $j \in \Z^+$ on the event $\{\X \mbox{ visits $j+\Z^+$ infinitely often}\}$ we have that the restriction of $\X$ to $j+\Z^+$ coincides with the extension on that set. Hence
$$ \{L(j,\infty)=\infty, L(i,\infty)<  \infty, \forall i\ge j+1\} = \{\X \mbox{ visits $j+\Z^+$ infinitely often}\}\cap A_j.$$
Hence, for any $j \in \Z^+$, we have 
$$
\begin{aligned}
\P\left( L(j,\infty)=\infty, L(i,\infty)<  \infty, \forall i\ge j+1 \right)  &\le \P(A_j) = \P(A_0).
\end{aligned}
$$
It follows that \eqref{eq:linf1} holds, if \eqref{eq:linf} does. Using Birkhoff's ergodic theorem (see e.g. Theorem 7.2.1 in \cite{Durrett2010}), we have
$$ \lim_{N \ti} \frac 1N \sum_{i=1}^N \1_{A_i} = \lim_{N \ti} \frac 1N \sum_{i=1}^N g(\Xi_i)= \P(A_0) > 0, \qquad \mbox{a.s..}$$
Consequently,  with probability 1, there exists $j\in 
\Z^+$ such that $A_j$ holds. Note that for each $j\ge 1$, $\{\X \text{ is transient}\}\subset \{\X^{(j)} \text{ is transient}\} \subset A_j^c$. Hence 
$$\P(\X \text{ is transient})\le  \P\left(\bigcap_{j=1}^{\infty} A_j^c\right)=0.$$
\end{proof}
From now on, in order to prove that $\X$ is not transient, we can assume that for all $j$
\begin{equation}\label{eq:assump} \P\left(L(j,\infty)=\infty, L(k,\infty)<  \infty, \forall k\ge j+1 \right)=0.
\end{equation}
 \begin{proposition}\label{independence}
Assuming \eqref{eq:assump}, for any $i \in \N$ we have that the random variable $L(i,\infty)$ and the extension of $\X$ to the set $\{0, 1, 2, \ldots, i\}$ are independent. {Furthermore, $(L(i,\infty))_{i\in \Z^+}$ is a sequence of identically distributed random variables.}
 \end{proposition}
 \begin{proof}
{For any $i \in  \N$, using  \eqref{eq:assump} we infer that $\X$ cannot localize on $\{0,1, \ldots, i\}$. In other words, visits $i+\N$ infinitely often.} Recall the definition of  $\X^{(i)}$ given in the proof  of Proposition~\ref{prop:no1} and its local times $(L^{(i)}(j, t):j \ge i, t\in [0, \infty))$. Using \eqref{eq:assump} with $j\le i-1$, we have that $\X$ must  almost surely spend an unbounded amount of time on the set  $i + \Z^+$. By the restriction principle, we have that $L(i, \infty)= L^{(i)}(i, \infty)$ almost surely. {Since $L^{(i)}(i, \infty)$ has the same distribution as $L(0,\infty)$, it follows that $(L(i,\infty))_{i\in \Z^+}$ are identically distributed random variables}. {Recall that $\mathcal{P}_{[0,i]}$ and $\mathcal{P}_i$ are $\sigma$-algebras defined respectively in \eqref{Pij} and \eqref{Pi}}. Notice that $L^{(i)}(i, \infty)$ is measurable with respect to  {$\mathcal{P}_{i}$}. On the other hand,  the extension of $\X$ to the set $\{0, 1, 2, \ldots i\}$ is measurable with respect to $\mathcal{P}_{[0,i]}$. As ${\mathcal{P}_{i}}$ and $\mathcal{P}_{[0,i]}$ are independent, this concludes the proof.
 \end{proof}

\begin{proposition}\label{prop:no2}
Assume  \eqref{eq:assump}. If $\P(L(1,\infty)=\infty)>0$ then $\P( \X \text{ is transient})=0$.
\end{proposition}
\begin{proof}
Suppose that  $\P(L(1,\infty)=\infty)>0$.  Recall the extensions $\X^{\ssup i}$ defined in the proof of Proposition~\ref{prop:no1}.  We have that $(\1_{\{L^{\ssup i}(i+1, \infty)=\infty\}})_{i\in \Z^+}$ is ergodic. More precisely, there exists a measurable function $f$ such that for all $i\in \Z^+$, we have  $\1_{\{L^{\ssup i}(i+1, \infty)=\infty\}}= f(\Xi_i)$, where $\Xi_i$ were defined in the proof of Proposition~\ref{prop:no1}. Recall that $\X^{\ssup 0}= \X$.  Using Birkhoff's ergodic theorem,  we have
$$\lim_{N \ti} \frac{1}{N} \sum_{i=0}^N\1_{\{L^{(i)}(i+1,\infty)=\infty\}}= \P(L(1,\infty)=\infty)>0 \quad \text{a.s.}.$$
Hence, with probability 1, there exists a $j\in \Z^+$ such that $L^{(j)}(j+1,\infty)=\infty$ yielding that $\X^{(j)}$ is not transient. The process $\X$ is thus not transient.
\end{proof}
 {Recall that $\widetilde{\X}$ (defined in Subsection~\ref{subsec:ex}) is the extension of $\X$ on $\{0,1\}$ and $(\widetilde{L}(i,t), i\in\{0,1\},t\in\R^+)$ is the process of local times  associated with $\widetilde{\X}$.}
 
\begin{proposition}\label{prop:no3}
{Assuming \eqref{eq:assump}, we have that $\widetilde{L}(1, \infty) = \infty$ almost surely.}
\end{proposition}
 
\begin{proof}
{Suppose that} $\P(\widetilde{L}(1,\infty)<\infty)>0$. {It follows that there exists a deterministic constant $K\in (1,\infty)$ such that $\P(\widetilde{L}(1,\infty)\le K)>0$. On the other hand, $$\P(L(1,\infty)\ge 2K)\ge \P\left( 1+\frac{\chi^{(1,0)}_1}{w(1+\chi^{(0,1)}_1)} \wedge \chi^{(1,2)}_1\ge 2K\right)>0.$$ 
Here we note that $\big(\chi^{(1,0)}_1/w(1+\chi^{(0,1)}_1) \big)\wedge \chi^{(1,2)}_1$ is the first waiting time at 1 before jumping to either 0 or 2. This random variable has support on the whole interval $[0,\infty)$ as $\chi^{(1,0)}_1$, $\chi^{(0,1)}_1$ and $\chi^{(1,2)}_1$ are independent exponential random variables.} 

{By the restriction principle, we have that  a.s. $L(1,\infty)\le \widetilde{L}(1,\infty)$. In virtue of Proposition \ref{independence}, we also have that $L(1,\infty)$ and $\widetilde{L}(1,\infty)$ are independent. Hence,  
$$0= \P(\widetilde{L}(1,\infty)\le K, L(1,\infty)\ge 2 K )=\P(\widetilde{L}(1,\infty)\le K)\P(L(1,\infty)\ge 2K)>0.$$
This contradiction ends the proof of the proposition. }  
\end{proof}

\begin{remark}\label{rem:no4} Notice that Proposition~\ref{comparison} implies  that 
$$
\P(\widetilde{L}(0, \infty) < \infty) \le \P(\widetilde{L}(1, \infty)  < \infty).
$$
{
In virtue of Proposition~\ref{prop:no2}, assuming \eqref{eq:assump} we have $\P(\widetilde{L}(0, \infty) < \infty)=\P(\widetilde{L}(1, \infty) < \infty)=0$. On the other hand, assuming that \eqref{eq:assump} does not hold, by reason of Proposition~\ref{prop:no1} the process $\X$ cannot be transient. Hence, to rule out the transience of $\X$, we assume from now on that}
\begin{equation}\label{eq:assumpov}
 \P(\widetilde{L}(1,\infty)=\widetilde{L}(0,\infty)=\infty)=1.
\end{equation}
\end{remark}
In what follows, $\P$ and $\E$ denote the probability measure and the expected value associated with the canonical process VRJP$(w)$ {on $\Z^+$} when $X_0=0$ and $\ell_x=1$ for all {$x\in \mathbb{Z}^+$}. We use $\P^{a,b}$ and  $\E^{a, b}$ to denote the probability measure and the expectation associated with the  canonical process extension on $\{0,1\}$ with initial local times $\ell_0=a$, $\ell_1=b$. Moreover, we use  $\E^{1,1+\exp}$  for the case when $\ell_0=1$ and $\ell_1$ is $1$ plus an exponential random variable with parameter 1. 
 
{Recall that $\widetilde\X$ and $\X^*$ ({defined in Subsection~\ref{subsec:ex}}) are two coupled VRJP on $\{0,1\}$ starting from 0 with initial local times $(\ell_0, \ell_1)$ and $(\ell_0^*, \ell_1^*)$ respectively. From now on, we always assume that $$\P(\ell_0=\ell_1=1)=\P(\ell_0^*=1,\ell_1^*=1+A)=1,$$  where $A$ is an exponential(1) random variable and it is {independent of everything else}. Note from Proposition~\ref{comparison}(ii) that $\widetilde{\X}$ under the probability measure $\P^{1,1+\exp}$ has the same law as $\X^*$ under $\P$.} {For $t\ge 0$, set 
$$\widetilde{\mathcal{F}}_t= \sigma\left(L(1,\infty),(\widetilde{X}_s)_{0\le s\le t}\right)\quad\text{and}\quad \mathcal{F}^*_t=\sigma\Big(L(1,\infty), ({X}^*_s)_{0\le s\le t}\Big).$$  
Assume \eqref{eq:assump}. In virtue of Proposition \ref{independence}, we have that $L(1,\infty)$ is independent of $\widetilde{\mathbf X}=(\widetilde{X}_t)_{t\ge0}$ and $\X^*=(X_t^*)_{t\ge0}$.
Hence} for some fixed $T>0$, it is clear (from the memoryless property) that when $\widetilde{\mathcal{F}}_T$ is given, the process $(\widetilde{X}_{t+T})_{t\ge 0}$ is a VRJP starting from $\widetilde X_T$ with initial local times $\widetilde L(0,T)$ and $\widetilde L(1,T)$. This property also holds for $\X^*$.   
 
In what follows, we make the convention that $\inf \varnothing=\infty.$ For $i\in\{0,1\}$, we define
\begin{align*} \widetilde{\eta}_i = \inf\{t\ge0 \colon \widetilde{L}(i, t)= L(1,\infty)\}, \quad
 \eta^*_i= \inf\{t\ge0 \colon L^*(i, t)= L(1,\infty)\}. 
\end{align*}

 {Since $L(1,\infty)$ is independent of both $\widetilde{\X}$ and $\X^*$, we have that} $\widetilde\eta_0$ and $\widetilde\eta_1$ are stopping times w.r.t. $\widetilde{\mathcal{F}}_t$ while $\eta_0^*$ and $\eta_1^*$ are stopping times w.r.t. $\mathcal{F}^*_t$. 
Set
$$\widetilde{\mathcal{F}}_{\widetilde{\eta}_0}:=\{B\in \widetilde{\mathcal{F}}: B\cap \{\widetilde\eta_0\le t\} \in \widetilde{\mathcal{F}}_t\},\quad \mathcal{F}^*_{{\eta}^*_0}:=\{B\in \mathcal{F}^*: B\cap \{\eta_0^*\le t\} \in \mathcal{F}^*_t\},$$ where $\widetilde{\mathcal{F}}$ and $\mathcal{F}^*$ are  the $\sigma$-fields generated by $\bigcup_{t\ge0}\widetilde{\mathcal{F}}_t$ and $\bigcup_{t\ge0}\mathcal{F}^*_t$  respectively.  

\begin{proposition}\label{prop.eta} {Assume \eqref{eq:assump}, \eqref{eq:assumpov} and $\P\left( L(1,\infty)<\infty\right)=1$. We have that $\eta^*_i<\infty$ and $\widetilde{\eta}_i<\infty$ a.s. for all $i\in\{0,1\}$ and}
\begin{align}\label{monofxi1}
  \{\eta^*_1>  \eta^*_0\} & \subseteq  \{\widetilde{\eta}_1>  \widetilde{\eta}_0\},\\
\label{monofxi0}
 \widetilde{L}(0,   \widetilde{\eta}_1)  & \ge L^*(0,  \eta^*_1),\\
\label{monofxi0b}
 \widetilde{L}(1,   \widetilde{\eta}_0)  & \le L^*(1,  \eta^*_0).\quad \end{align}
Furthermore, 
\begin{align}
{\label{sper}  \1_{\{\widetilde{\eta}_1>  \widetilde{\eta}_0\}} \E\left[ \widetilde{L}(0, \widetilde{\eta}_1)\wedge M-L(1,\infty)\wedge M \big| \  \widetilde{\mathcal{F}}_{\widetilde{\eta}_0} \right]  \ge \1_{\{{\eta}^*_1>  {\eta}^*_0\}} \E\left[ L^*(0, \eta^*_1)\wedge M-L(1,\infty)\wedge M \big| \ \mathcal{F}^*_{\eta^*_0}\right]}
\end{align}
for any deterministic $M>0$.
\end{proposition}
\begin{proof}
{We first notice that the process $\X^*$ can be also interpreted as a VRJP on $\{0,1\}$ starting from 1 with initial local times one.  Using symmetry, it follows from \eqref{eq:assumpov} that $$\P(L^*(1,\infty)=L^*(0,\infty)=\infty)=1.$$
{Since $L(1,\infty)<\infty$ a.s., we thus} have that $\eta^*_i<\infty$ and $\widetilde{\eta}_i<\infty$ a.s. for all $i\in\{0,1\}$. For $i\in\{0,1\}$, set $$\widetilde{H}_i=\inf\{n\ge 0: \widetilde{L}(i, \widetilde \tau_n)\ge L(1,\infty)\}\quad \text{and}\quad H^*_i=\inf\{n\ge 0: {L}^*(i,  \tau_n^*)\ge L(1,\infty)\}.$$ 
By reason of Proposition~\ref{comparison}(i),  we have 
$\widetilde{L}(0,\widetilde \eta_1)=\widetilde{L}(0,\widetilde \tau_{\widetilde H_1}) 
\ge L^*(0,  \tau_{H^*_1}^*)=L^*(0, \eta^*_1).$
Similarly, we also have
$
\ \widetilde{L}(1,\widetilde \eta_0)=\widetilde{L}(1,\widetilde \tau_{\widetilde H_0}) 
 \le L^*(1,  \tau_{H^*_0}^*)=L^*(1, \eta^*_0).
$
Thus, \eqref{monofxi0}-\eqref{monofxi0b} are proved. On the other hand, 
$$\widetilde{\eta}_0
= L(1,\infty)+L(1,\widetilde \eta_0) -2 \,{\le}\, 
  L(1,\infty)+L^*(1,\eta^*_0)-2-A={\eta}_0^*,
$$
It is also clear that $\widetilde{\eta}_1 \,{\ge}\, \eta^*_1$. Hence,  for each $\omega \in \{\eta^*_1 > \eta^*_0\}$ we have that $\widetilde{\eta}_1(\omega) \ge {\eta}_1^*(\omega)>{\eta}_0^*(\omega) \ge \widetilde{\eta}_0(\omega)$. This verifies \eqref{monofxi1}.
}

For $t\ge0$, set $\widetilde{Y}_t:=\widetilde{X}_{t+\widetilde\eta_0}$ and $Y_t^*:=X^*_{t+\widetilde\eta_0^*}$. Given $\widetilde{\mathcal{F}}_{\widetilde\eta_0}$ and ${\mathcal{F}}^*_{\eta^*_0}$ respectively, the processes  $\widetilde{\Y}=(\widetilde{Y}_t)_{t\ge0}$ and $\Y^*=({Y}^*_t)_{t\ge0}$ are VRJP on $\{0,1\}$ with $\widetilde Y_0=Y^*_0=0$. {Notice that $\widetilde{Y}$ and $\widetilde{Y}^*$ can be reconstructed using exponential families, denoted by $\widetilde{\boldsymbol{\chi}}$ and $\boldsymbol{\chi}^*$ respectively, which are independent of their initial local times.} Denote by $(L_{\widetilde{Y}}(i,t),i\in\{0,1\})_{t\ge0}$ and $(L_{Y^*}(i,t),i\in\{0,1\})_{t\ge0}$ the processes of local times associated with $\widetilde{\Y}$ and $\Y^*$ respectively. Note that $L_{Y^*}(0,0)=L_{\widetilde Y}(0,0)=L(1,\infty)$ and $L_{Y^*}(1,0)=L^*(1,\eta^*_0)\ge L_{\widetilde Y}(0,0)=\widetilde L(1,\widetilde \eta_0)$. {We also have 
$\widetilde\eta_1-\widetilde\eta_0=\inf\{u:L_{\widetilde Y}(1,u)=L_{\widetilde Y}(0,0)\}$ and $\eta_1^*-\eta_0^*=\inf\{u:L_{ Y^*}(1,u)=L_{Y^*}(0,0)\}$. Hence, there exists a deterministic Borel measurable function $\varphi$ such that 
\begin{align*} \1_{\{\widetilde{\eta}_1>  \widetilde{\eta}_0\}} \big(\widetilde{L}(0,\widetilde \eta_1)\wedge M-L(1,\infty)\wedge M \big) & = \varphi(\widetilde{\boldsymbol{\chi}},L_{\widetilde{Y}}(0,0), L_{\widetilde{Y}}(1,0) ),\\ \1_{\{{\eta}^*_1>  {\eta}^*_0\}}\big( L^*(0,\eta_1^*)\wedge M- L(1,\infty)\wedge M \big) &= \varphi(\boldsymbol{\chi}^*,L_{Y^*}(0,0), L_{Y^*}(1,0)). \end{align*}
As a consequence of \eqref{monofxi1}-\eqref{monofxi0}, we have that  $\varphi(\widetilde{\boldsymbol{\chi}},L_{\widetilde{Y}}(0,0), L_{\widetilde{Y}}(1,0) ) \ge \varphi(\boldsymbol{\chi}^*,L_{Y^*}(0,0), L_{Y^*}(1,0))$. We also note that $\widetilde{\boldsymbol{\chi}}$ and $\boldsymbol{\chi}^*$ has the same distribution.
Hence, 
\begin{align}\label{sper0}
\nonumber  & \E\left[\1_{\{\widetilde{\eta}_1>  \widetilde{\eta}_0\}}\left(\widetilde{L}(0,\widetilde\eta_1)\wedge M-L(1,\infty)\wedge M\right)\ \big|\ \widetilde{\mathcal{F}}_{\widetilde\eta_0}\right]  \\
\nonumber & =\E\Big[\varphi(\widetilde{\boldsymbol{\chi}},L_{\widetilde{Y}}(0,0), L_{\widetilde{Y}}(1,0) )\ |\ {{L_{\widetilde{Y}}(0,0), L_{\widetilde{Y}}(1,0)}}\Big]\\ & \nonumber \ge  \E\Big[\varphi(\boldsymbol{\chi}^*,L_{Y^*}(0,0), L_{Y^*}(1,0))\ |\ {{L_{Y^*}(0,0), L_{Y^*}(1,0)}}\Big]
\\
  & = \E\left[\1_{\{{\eta}^*_1>  {\eta}^*_0\}}\Big( L^*(0,\eta^*_1)\wedge M-L(1,\infty)\wedge M\Big)
\ \big|\ {\mathcal{F}}^*_{\eta^*_0}\right].
 \end{align}}
Note that $\1_{\{\widetilde\eta_1 >\widetilde\eta_0\}}$ is $\widetilde{\mathcal{F}}_{\widetilde{\eta}_0}$-measurable as $\widetilde\eta_0$ and $\widetilde\eta_1$ are stopping times w.r.t. $(\widetilde{\mathcal{F}}_t)_{t\ge0}$. Similarly, $\1_{\{{\eta}^*_1>  {\eta}^*_0\}}$ is ${\mathcal{F}}^*_{\eta^*_0}$-measurable. Combining this fact with \eqref{sper0}, we obtain \eqref{sper}.
\end{proof}

\begin{proposition}\label{ineq.L}
{Let $\Y=(Y_t)_{t\ge0}$ be a VRJP$(\ell,w)$ extension on $\{0,1\}$ {starting from $0$} with local time process $(L_{\Y}(0,t), L_{\Y}(1,t))_{t\ge 0}$ and positive initial local times $\ell_0=a, \ell_1=b$. Let $\E^{a,b}$ be the expected value associated with $\Y$ and set $$\xi(t) = \inf\{u \colon {L}_{\Y}(1, u) =t\}.$$ For any $a>b>1$, there exists a positive constant $\rho(a,b)$ depending on $a$ and $b$ such that}
\begin{equation}\label{s1s01}
 \E^{a, b}[ L_{\Y}(0, \xi(a))] - a = a -b+\rho(a,b).
\end{equation}
\end{proposition}
\begin{proof}
{ For $a>b >1$ and $t \ge b$, we first claim that
\begin{align}\label{eq.mean}m_{a,b}(t):= \E^{a,b}[{L}_{\Y}(0, \xi(t))] = a  + \int_{b}^t \frac{\E^{a,b}[w({L}_{\Y}(0, \xi(u)))]}{w(u)} \rmd u.\end{align}
Indeed, we have  ${L}_{\Y}(0, \xi(t+\rmd t))={L}_{\Y}(0, \xi(t))+\nu \psi$, where $\nu$ is $\text{Bernoulli}\left(w({L}_{\Y}(0, \xi(t)))\rmd t\right)$ and $\psi$ is $\text{exponential}(w(t))$, and $\nu$ and $\psi$ are independent of each other given ${L}_{\Y}(0, \xi(t))$. Hence,
\begin{align*}m_{a,b}(t+dt)-m_{a,b}(t)=\E^{a,b}\left[\E^{a,b}\left[\nu\psi\ \big|\ {L}_{\Y}(0, \xi(t))\right]\right]&=\E^{a,b}\left[ \E^{a,b}[\psi]w({L}_{\Y}(0, \xi(t)))\rmd t\right]
\\
& = \frac{\E^{a,b}[w({L}_{\Y}(0, \xi(t)))]}{w(t)}\rmd t,
\end{align*}
yielding \eqref{eq.mean}.
Denote by $\P^{a,b}$ the probability measure associated with $\Y$. Since $Y_0=0$, we must have ${L}_{\Y}(0, \xi(u))> {L}_{\Y}(0, \xi(b)) = a$, $\P^{a,b}$-a.s., for all $u >b$. Hence, for $u \in (b, a)$ we obtain that $w({L}_{\Y}(0, \xi(u)))> w(a) > w(u)$, $\P^{a,b}$-a.s.. As a result,
$$\rho(a,b):= \int_{b}^a\frac{\E^{a,b}[w({L}_{\Y}(0, \xi(u)))-w(u)]}{w(u)}{\rmd u}>0$$
and thus
$$  \E^{a,b}[{L}_{\Y}(0, \xi(a))] - a  = m_{a,b}(a){-a}=  \int_{b}^a \frac{\E^{a,b}[w({L}_{\Y}(0, \xi(u)))]}{w(u)} \rmd u = a -b+\rho(a,b).$$
}
\end{proof}

\begin{proof}[Proof of Theorem~\ref{th:notransient}] {In the following, we focus on the extension of VRJP$(w)$ on $\Z^+$. We prove that this process cannot be transient almost surely. This implies that also the extension on $\Z^-$ cannot be transient. Hence, by the restriction principle, the result extends to VRJP$(w)$ on $\Z$. We  reason  by  contradiction and for simplicity denote by $\X$ a VRJP on $\Z^+$}. 

Suppose $\P(\X\text{ is transient})>0$. By virtue of Propositions \ref{prop:no1},  \ref{prop:no2} and  \ref{prop:no3} and Remark~ \ref{rem:no4}, we have that \eqref{eq:assump} is satisfied and that
 \begin{equation}\label{eq:asslt}
 \P(\widetilde{L}(1,\infty)=\widetilde{L}(0,\infty)=\infty)=1 \qquad \mbox{ and }   \qquad  \P\left( L(1,\infty)<\infty\right)=1.
 \end{equation}
 For $i \in \{0,1\}$, set  
$$
S_i = L(i, \infty)\wedge M,
$$ 
where $M>0$ is a large deterministic parameter that we will choose later. {We are going to prove that assuming \eqref{eq:assump} and \eqref{eq:asslt} there exists a large deterministic $M>0$ such that }
\begin{equation}\label{eq:contr0}
\E[ S_0]> \E[S_1].
\end{equation}
{As $S_0$ and $S_1$ are identically distributed and bounded,  \eqref{eq:contr0} yields a contradiction.

The rest of this section is devoted to prove  \eqref{eq:contr0}, and we assume throughout that \eqref{eq:assump} and \eqref{eq:asslt} hold.} 

{Let $\P^{(rand)}$ and $\E^{(rand)}$ be respectively the probability measure and the expectation associated with the canonical process extension on $\{0,1\}$, with initial local times one, but randomized starting point, i.e.\ either $0$ or $1$ with probability $1/2$.}
{Recall from Proposition~\ref{prop.eta} that for each $i\in\{0,1\}$, $\widetilde{\eta_i}<\infty$, ${\eta^*_i}<\infty$ a.s. and note that $\P^{(rand)}(\widetilde{\eta}_0 < \widetilde{\eta}_1) = \frac12$ due to symmetry.} 

By the restriction principle, $\widetilde\eta_1$ is {equal} to the last time that $\X$ {stays} at $1$, i.e. $\widetilde\eta_1=\sup\{t: \X_t=1 \}.$ 
As a result, 
\begin{equation}
\label{L0eta1}
L(0,\infty)= \widetilde{L}(0,\widetilde{\eta}_1).
\end{equation}
We also notice from Proposition \ref{independence} that $S_1=L(1,\infty)\wedge M$ is independent of the filtrations $(\widetilde{\mathcal{F}}_t)_{t\ge0}$ and $({\mathcal{F}}^*_t)_{t\ge0}$. Using \eqref{sper}, \eqref{L0eta1}, Proposition \ref{comparison}(ii) and the fact that $\{\widetilde\eta_1 >\widetilde\eta_0\}\in \widetilde{\mathcal{F}}_{\widetilde{\eta}_0}$ (as $\widetilde\eta_0$ and $\widetilde\eta_1$ are stopping times w.r.t. $(\widetilde{\mathcal{F}}_t)_{t\ge0}$), we thus get
$${
\begin{aligned}
 \E\left[( S_0 -   S_1)\1_{\{\widetilde{\eta}_1> \widetilde{\eta}_0\}}\right] &= \E\left[\E\left[ (S_0 -   S_1)\1_{\{\widetilde{\eta}_1> \widetilde{\eta}_0\}} \ \big|\  \widetilde{\mathcal{F}}_{\widetilde{\eta}_0}\right]\right]\\
 & \stackrel{\eqref{L0eta1}}{=}  \E\left[\E\left[ \widetilde{L}(0, \widetilde{\eta}_1)\wedge M -   S_1\ \big|\ \widetilde{\mathcal{F}}_{\widetilde{\eta}_0}\right]\1_{\{\widetilde{\eta}_1> \widetilde{\eta}_0\}}\right]\\
 & \stackrel{\eqref{sper}}{\ge} \E\left[\E\left[ L^*(0, \eta^*_1)\wedge M -   S_1\ \big|\  {\mathcal{F}}^*_{{\eta}^*_0}\right]\1_{\{\eta^*_1> \eta^*_0\}} \right]\\
  &= \E^{1,1+\exp}\left[\E\left[ \widetilde{L}(0, \widetilde{\eta}_1)\wedge M  -   S_1 \ \big|\  \widetilde{\mathcal{F}}_{\widetilde{\eta}_0} \right]\1_{\{\widetilde{\eta}_1> \widetilde{\eta}_0\}}\right]. 
 \end{aligned}
 }$$
 {To obtain the last identity, we use the fact (from  Proposition~\ref{comparison}(ii)) that $\widetilde{\X}$ under the probability measure $\P^{1,1+\exp}$ has the same law as $\X^*$ under $\P$.}
 Hence,
{\begin{equation}\label{Eiq1}
 \E[( S_0 -   S_1)\1_{\{\widetilde{\eta}_1> \widetilde{\eta}_0\}}] \ge \E^{\ssup{rand}}\left[\E\left[  \widetilde{L}(0, \widetilde{\eta}_1)\wedge M -   S_1\ \big|\  \widetilde{\mathcal{F}}_{\widetilde{\eta}_0}  \right]\1_{\{\widetilde{\eta}_1> \widetilde{\eta}_0\}}\right].
\end{equation} }
We also have 
$${
\begin{aligned} 
\E[(S_0 - S_1)\1_{\{\widetilde{\eta}_1< \widetilde{\eta}_0\}}] & \stackrel{\eqref{L0eta1}  }{=}   - \E[(S_1-\widetilde{L}(0, \widetilde{\eta}_1)\wedge M ) \1_{\{\widetilde{\eta}_1< \widetilde{\eta}_0\}}] \\
& \stackrel{\eqref{monofxi0}+\eqref{monofxi1}}{\ge} - \E[(S_1-L^*(0, \eta^*_1) \wedge M )\1_{\{\eta^*_1< \eta^*_0\}}]\\
& \stackrel{\rm\tiny Prop.~\ref{comparison}(ii)}{=} - \E^{1,1+ \exp}[(S_1-\widetilde{L}(0, \widetilde{\eta}_1) \wedge M)\1_{\{\widetilde{\eta}_1< \widetilde{\eta}_0\}}]
\end{aligned}
}$$
and thus
\begin{align}\label{Eiq2}
\nonumber \E[(S_0 - S_1)\1_{\{\widetilde{\eta}_1< \widetilde{\eta}_0\}}] 
 & \ge -\E^{\ssup{rand}}[(S_1-\widetilde{L}(0, \widetilde{\eta}_1) \wedge M )\1_{\{\widetilde{\eta}_1< \widetilde{\eta}_0\}}] \\
 &=  -\E^{\ssup{rand}}[(S_1-\widetilde{L}(1, \widetilde{\eta}_0) \wedge M)\1_{\{\widetilde{\eta}_1> \widetilde{\eta}_0\}}]
 \end{align}
where the second equality is a consequence of symmetry.

By the definition of $\widetilde\eta_1$ and $\widetilde\eta_0$, we have $\P(\widetilde\eta_1=\widetilde\eta_0<\infty)=0$. Combining (\ref{Eiq1}) and (\ref{Eiq2}), we obtain
{\begin{equation}\label{s1s0}
\E[(S_0 - S_1)] \ge \E^{\ssup{rand}}\left[\left(\E[ \widetilde{L}(0, \widetilde{\eta}_1)\wedge M -   S_1\ \big|\  \widetilde{\mathcal{F}}_{\widetilde{\eta}_0}] -  S_1 + \widetilde{L}(1, \widetilde{\eta}_0)\wedge M \right)\1_{\{\widetilde{\eta}_1> \widetilde{\eta}_0\}}\right].
\end{equation}
}
{We are now going to apply Proposition~\ref{ineq.L} to the process $\Y=({Y}_{t})_{t\ge0}$ defined by ${Y}_{t}=\widetilde{X}_{t+\widetilde \eta_0}$ for all $t\ge0$ and $a:=L(1,\infty)=\widetilde L(0,\widetilde\eta_0)=\widetilde L(1,\widetilde \eta_1), b:=\widetilde{L}(1,\widetilde{\eta_{0}})$ to complete the proof. Notice that on the event $\{\widetilde{\eta}_1> \widetilde{\eta}_0\}\in \widetilde{\mathcal{F}}_{\widetilde \eta_0}$, we have $a>b$ and
\begin{align}\label{eq.trunc} \nonumber
  & \E[ \widetilde{L}(0, \widetilde{\eta}_1)\wedge M -   S_1 \big|  \widetilde{\mathcal{F}}_{\widetilde{\eta}_0}] -  S_1+ \widetilde{L}(1, \widetilde{\eta}_0)\wedge M\\ & = \E^{a, b}[{L_{\Y}(0, \xi(a))\wedge M}- a\wedge M]-(a\wedge M-b\wedge M)
\end{align}
in which we note that $\xi(a):=\inf\{u: L_{\Y}(1,u)=a\}= \inf\{u: \widetilde{L}(1,u+\widetilde\eta_0)=L(1,\infty)\}=\widetilde{\eta_1}-\widetilde\eta_0$ and that $L_{\Y}(0,\widetilde{\eta_1}-\widetilde\eta_0 )=\widetilde{L}(0,\widetilde{\eta_1})$.}

{By virtue of Proposition~\ref{ineq.L}, there exists a positive constant $\rho(a,b)$ (depending only on $a$ and $b$) such that
$$\E^{a,b}[{L}_{\Y}(0, \xi(a))] - a= a-b  +\rho(a,b).$$}
On the other hand, by the monotone convergence theorem, $\E^{a, b}[L_{\Y}(0, \xi(a))\wedge M]-a\wedge M$ converges to $\E^{a, b}[L_{\Y}(0, \xi(a))]-a$ as $M\to\infty$. Hence, using $a>b$
\begin{equation}\label{truncated}
 \E^{a, b}[L_{\Y}(0, \xi(a))\wedge M] - a\wedge M =: a-b+\rho_M(a,b)\stackrel{(a>b)}{\ge} a\wedge M -b\wedge M+\rho_M(a,b),
\end{equation}
where $(\rho_M(a,b))_{M \ge0}$ is a increasing sequence  such that $\rho_ M(a,b)\uparrow \rho(a,b)$ as $M\to\infty$. 
 We can choose $M$ and $N$ large enough such that 
\begin{align}\label{prob.eta}
\P^{(rand)}\left(\widetilde{\eta}_1>\widetilde{\eta}_0, \rho_M(a,b)>\frac{1}{N}\right)>\frac{1}{2}\P^{(rand)}\left(\widetilde{\eta}_1> \widetilde{\eta}_0\right)\stackrel{\rm (symmetry)}{=}\frac{1}{4}.
\end{align}
Combining (\ref{s1s0}) with \eqref{eq.trunc}, \eqref{truncated} and \eqref{prob.eta}, we get
$$\E[S_0-S_1]\ge \E^{(rand)}\left[\rho_M(a,b) \1_{\left\{\widetilde \eta_1>\widetilde \eta_0,\ \rho_M(a,b)>\frac{1}{N}\right\} }\right] > \frac{1}{4N}>0$$
and this proves  \eqref{eq:contr0}.
\end{proof} 

\section{A martingale approach to VRJP on two vertices}

\subsection{Some preliminary classical results for finite variation processes}
In this subsection, we briefly recall some useful results for c\`adl\`ag finite variation processes and martingales (see e.g. \cite{Fima, Protter2004}) that is needed in our analysis. {\bf The reader can skip the section and refer to it when the results are recalled.} 

Recall that a process is said to have \textit{finite variation} if it has bounded variation on each finite time interval with probability one. Let ${\bf \Theta}=(\Theta_t)_{t\ge0}$ be a c\`adl\`ag finite variation processes. For each $t\ge0$, let $\Theta_{t-}=\lim_{s\uparrow  t}\Theta_s$ and $\Delta \Theta_t=\Theta_t-\Theta_{t-}$ be respectively the left limit and the size of the jump of ${\bf \Theta}$ at time $t$. We will make use of the following results throughout this section:

\noindent{\bf 1. (Integration by parts formula)} Let ${\bf \Theta}=(\Theta_t)_{t\ge0}$ and ${\bf \Upsilon}=(\Upsilon_t)_{t\ge0}$ be two c\`adl\`ag finite variation processes in $\mathbb{R}$. Then for $0 \le s \le t$,
\begin{align}\label{ibpart}
\Theta_t\Upsilon_t-\Theta_s\Upsilon_s=\int_s^t \Theta_{u-}\text{d}\Upsilon_u+\int_s^t \Upsilon_{u-}\text{d}\Theta_u+[\Theta,\Upsilon]_t-[\Theta,\Upsilon]_s,
\end{align}
where the integrals are in the sense of Lebesgue-Stieltjes and we denote by
 $[{\bf \Theta},{\bf \Upsilon}]$ the \textit{covariation} of ${\bf \Theta}$ and ${\bf \Upsilon}$, given by $$[\Theta,\Upsilon]_t=\sum_{0<u\le t}\Delta \Theta_u \Delta \Upsilon_u.$$ It immediately follows that
\begin{align}\label{changvar}
\Theta_t^2-\Theta_s^2=2\int_s^t \Theta_{u-}\text{d}\Theta_u+[\Theta]_t-[\Theta]_s,
\end{align}
where $[{\bf \Theta}]=[{\bf \Theta},{\bf \Theta}]$ stands for the \textit{quadratic variation} of ${\bf \Theta}$, i.e. covariation of ${\bf \Theta}$ and itself.

%
%

\noindent{\bf 2. (Angle bracket)} Let ${\bf M}=(M_t)_{t\ge0}$ be a c\`adl\`ag square-integrable local martingale with finite variation. We denote by $\langle {\bf M}\rangle$ the \textit{angle bracket} of ${\bf M}$, i.e. the unique right-continuous predictable increasing process such that $\langle { M}\rangle_0=0$ and ${\bf M}^2-\langle {\bf M}\rangle$ is a local martingale. Note that $[{\bf M}]-\langle {\bf M}\rangle$ is also a local martingale. As a consequence, if ${\bf M}$ is a (true) martingale w.r.t the filtration $(\mathcal{F}_t)_{t\ge0}$ then
\begin{equation}\label{isometry}
\E[(M_t-M_s)^2|\mathcal{F}_s]=\E[M^2_t-M^2_s|\mathcal{F}_s]=\E[[M]_t-[M]_s|\mathcal{F}_s]=\E[\langle M\rangle_t-\langle M \rangle_s|\mathcal{F}_s]
\end{equation}
for all $t\ge s\ge 0$. Also note that {\bf M} is a martingale if for all $t\ge0$, $\E([M]_t)<\infty$. 

Let ${\bf H}$ be a locally bounded  $(\mathcal{F}_t)_{t\ge0}$-predictable process and denote by ${\bf H}\cdot {\bf M}$ the c\`adl\`ag local square-integrable martingale with finite variation defined by $(H\cdot M)_t=\int_0^t H_{s} \text{d}M_s$. Recall the following identities
\begin{equation}\label{M.integrak}
\langle H\cdot M\rangle_t=\int_0^t H^2_{s}\text{d}\langle M\rangle_s \quad \text{ and }\quad [H\cdot M]_t=\int_0^t H^2_{s}\text{d}[M]_s.\end{equation}
Recall also that ${\bf H}\cdot {\bf M}$ is a square integrable martingale if and only if for all $t>0$, $\E[\langle H\cdot M\rangle_t]<\infty$.

\noindent{\bf 3.} {\bf (Doob's maximal inequality)}, Let $T\ge t\ge 0$, $p\ge 1$ and ${\bf M}=(M_t)_{t\ge0}$ be a c\`adl\`ag martingale adapted to a filtration $(\mathcal{F}_t)_{t \ge 0}$ such that $\E[|M_T|^p  ]<\infty$. Then
\begin{equation}\label{doob.ineq}
\mathbb{P}\left[  \sup_{t \le s \le T} |M_s |\ge \lambda  \big| \mathcal{F}_t\right] \le \frac{\mathbb{E} [|M_T|^p  \big| \mathcal{F}_t ]}{\lambda^p}.\end{equation}

 \vspace{1cm}

In the next subsections we focus on the extension $\widetilde\X=(\widetilde X_t)_{t\ge0}$ of the VRJP$(w)$ $\X$ to $\{0, 1\}$. We assume that  $w$ satisfies Assumption W.

\subsection{A martingale decomposition for VRJP}
  
For $i\in \{0,1\}$ and $t\ge 0$, set $$W(i,t):=w(\widetilde{L}(i,t)),$$ where $\left(\widetilde{L}(i,t),i\in\{0,1\}\right)_{t\in[0,\infty)}$ is the local time process of the extension  $\widetilde\X$. For each $t\ge0$, we also define 
\begin{equation*}
H_t:=\int_{1}^{\widetilde{L}(0,t)}\frac{1}{w(u)}\text{d}u-
\int_{1}^{\widetilde{L}(1,t)}\frac{1}{w(u)}\text{d}u.
\end{equation*} 
Let $\mathbf{M}=(M_t)_{t\ge0}$ and $\mathbf{A}=(A_t)_{t\ge0}$ be processes given by
\begin{align}\label{eq:defofM} M_t& :=\1_{\{\widetilde X_t=1\}}-\int_{0}^t\left(W(1,u)\1_{\{\widetilde X_u=0\}}-W(0,u)\1_{\{\widetilde X_u=1\}}\right) \rmd u,\\
A_t& :=-\frac{1}{W(0,t) W(1,t)}.
\end{align}
It is clear $\mathbf{A}$ is a bounded non-decreasing c\`adl\`ag process. On the other hand, we have the following result which is a consequence of Lemma 4.1 in \cite{OM2018}. For the sake of completeness,  we include its proof in the Appendix.   
\begin{lemma}\label{martingale} $\mathbf{M}$ is a c\`adl\`ag martingale with finite variation. Furthermore, its angle bracket process is given by \begin{equation}\label{eq:anglebracket}\langle M \rangle_t=\int_0^t \Lambda_u\rmd u,\end{equation}
where for each $t\ge0$,
$$\Lambda_t:=\1_{\{\widetilde X_t=0\}}W(1,t)+\1_{\{\widetilde X_t=1\}}W(0,t).$$
\end{lemma}

\begin{proposition}\label{pr:deco} For each $t\ge 0$, we have the following decomposition 
\begin{equation}\label{Ht}
H_t=\frac{\1_{\{\widetilde X_t=1 \}}}{W(0,t)W(1,t)} 
+\int_0^t \1_{\{\widetilde X_{u-}=1 \}}{\rm d}A_u  -\int_0^t \frac{{\rm d}M_u}{W(0,u-) W(1,u-)}. 
\end{equation}
\end{proposition}
\begin{proof}
Notice that 
$$H_t=\int_{0}^{t}\frac{\1_{\{\widetilde  X_u=0 \}}}{W(0,u)}\text{d}u-\int_{0}^{t}\frac{\1_{\{ \widetilde X_u=1\}}}{W(1,u)}\text{d}u=\int_0^t\frac{\Lambda_u \text{d}u }{W(0,u)W(1,u)},$$
where the first equality is derived by a simple time change. {Applying the integration by parts (formula (\ref{ibpart})) to the processes $\mathbf{A}=(A_t)_{t\ge0}$ and $\mathbf{I}=(\1_{\{\widetilde X_{t}=1\}})_{t\ge0}$, we have
\begin{align*}\int_0^t \1_{\{\widetilde X_{u-}=1\}} \rmd A_u & = -\frac{\1_{\{\widetilde X_{t}=1\}}}{W(0,t)W(1,t)} +\int_0^t \frac{\rmd \big(\1_{\{\widetilde X_{u}=1\}}\big) }{W(0,u-)W(1,u-)}-[\mathbf{I},\mathbf{A}]_t,\\
& = -\frac{\1_{\{\widetilde X_{t}=1\}}}{W(0,t)W(1,t)} +\int_0^t \frac{\rmd M_u }{W(0,u-)W(1,u-)}+\int_0^t\frac{\Lambda_u\text{d}u}{W(0,u)W(1,u)}-[\mathbf{I},\mathbf{A}]_t,
\end{align*}
where in the last step we used \eqref{eq:defofM} and the definition of the process $\Lambda$. Hence,
$$H_t=\frac{\1_{\{\widetilde X_t=1 \}}}{W(0,t)W(1,t)} 
+\int_0^t \1_{\{\widetilde X_{u-}=1 \}}{\rm d}A_u  -\int_0^t \frac{{\rm d}M_u}{W(0,u-) W(1,u-)}+[\mathbf{I},\mathbf{A}]_t.$$
We finish the proof by showing that} $[\mathbf{I},\mathbf{A}]\equiv 0$. Indeed, if for some $t>0$, 
$$\Delta I_t \Delta A_t=\Delta 1_{\{\widetilde X_t=1\}}\Delta\left(\frac{-1}{w(\widetilde{L}(0,t))w(\widetilde{L}(1,t))}\right)\neq 0,$$
 then $t=\tau$, which is a jump time of $(\widetilde X_t)_{t\ge0}$. However, the distributions of $\widetilde L(0,\tau)$ and $\widetilde L(1,\tau)$ are absolutely continuous with respect to the Lebesgue measure. Thus, $\widetilde L(0,\tau)$ and $\widetilde L(1,\tau)$  have probability zero to take value in the set of discontinuous points of $w$ (which is at most countable). Hence, $\Delta A_{\tau}=0$ a.s.. This contradiction implies that $\Delta I_t \Delta A_t=0$ a.s. for all $t>0$ and thus $[\mathbf{I},\mathbf{A}]\equiv 0$.   
\end{proof} 

\subsection{Weakly VRJP on two vertices}
The following result is well-known in the case of discrete-time martingales (see e.g. \cite{Durrett2010}). We are interested in the continuous time version, and for the sake of completeness we include a proof in the Appendix.

\begin{lemma}\label{boundedjump} Let $(N_t)_{t\ge 0}$ be a c\`adl\`ag martingale with uniformly bounded jumps, i.e. there exits a non-random constant $C<\infty$ such that for all $t>0$, $|\Delta N_t|\le C$. Then
$$\P\left(\left\{\lim_{t\to\infty}N_t \ \text{exists and is finite}\right\} \cup \left\{ \liminf_{t\to\infty}N_t=-\infty\ \text{and}\   \limsup_{t\to\infty} N_t=+\infty  \right\}  \right)=1.$$
\end{lemma} 

Next result shows that the weakly VRJP$(w)$  defined on $\{0,1\}$ spends an infinite amount of time in both vertices. 
\begin{proposition}\label{VRJP2vert} {Under Assumption W and the further assumption  that} $$\int_1^{\infty}\frac{{\rm d}t}{w(t)}=\infty,$$ 
 we have that  $\widetilde{L}(0,\infty)=\widetilde{L}(1,\infty)=\infty$ {almost surely.}
\end{proposition} 
\begin{proof}
Recall  from Proposition~\ref{pr:deco} that we can decompose the process $(H_t)_{t\ge0}$ as follows 
\begin{align*}H_t &=  \frac{\1_{\{\widetilde X_t=1 \}}}{W(0,t)W(1,t)} +\int_0^t \1_{\{\widetilde X_{u-}=1 \}}\text{d}A_u - \int_0^t \frac{\text{d}M_u}{W(0,u-) W(1,u-)}\\
&= R_t-\widetilde{M}_t,
\end{align*}
where $$R_t:=\frac{\1_{\{\widetilde X_t=1 \}}}{W(0,t)W(1,t)} +\int_0^t \1_{\{\widetilde X_{u-}=1 \}}\text{d}A_u, \quad \widetilde{M}_t:=\int_0^t \frac{\text{d}M_u}{W(0,u-) W(1,u-)}.$$
{Note that $\sup_{t\ge 0}|R_t|<\infty$ almost surely since} 
$$ \left|\int_0^\infty \1_{\{\widetilde X_{u-}=1 \}}\text{d}A_u \right|\le A_{\infty}-A_0=\frac{1}{W(0,0)W(1,0)}=1.$$  On the other hand, as $W(i,t) \ge 1$ for $i\in\{0,1\}$ and for all $t\ge 0$, we have that  $$[\widetilde{M}]_t \stackrel{\eqref{M.integrak}}{=}\int_0^t\frac{1}{W(0,u-)^2W(1,u-)^2} \rmd [M]_u \le [M]_t<\infty$$ for each $t>0$ and $(\widetilde{M}_t)_{t\ge0}$ is thus a martingale. We also have $$|\Delta\widetilde{M}_t|=\frac{|\1_{\{\widetilde X_t=1\}}-\1_{\{\widetilde X_t-=1\}}|}{W(0,t-) W(1,t-)}\le 1.$$ 
Therefore, we have from Lemma~\ref{boundedjump} that either $(\widetilde M_{t})_{t\ge0}$ converges to a finite limit or it oscillates in the sense that $\liminf_{t\to\infty}\widetilde M_t=-\infty$ and $\limsup_{t\to\infty}\widetilde M_t=+\infty$.
Furthermore, \begin{align*}\left\{\widetilde{L}(0,\infty)\wedge\widetilde{L}(1,\infty)<\infty\right\}&\subseteq \left\{ \lim_{t\to\infty}H_t= +\infty\right\}\cup \left\{ \lim_{t\to\infty}H_t=-\infty\right\}\\& = \left\{ \lim_{t\to\infty}\widetilde M_t=+\infty \right\}\cup\left\{ \lim_{t\to\infty}\widetilde M_t=-\infty \right\}.\end{align*}
If  the process $(\widetilde M_t)_t$ converges, it cannot converge to neither $+ \infty$ nor $-\infty$,  i.e.
$$\P\left\{ \lim_{t\to\infty}\widetilde M_t=+\infty\right\} = \P\left\{ \lim_{t\to\infty}\widetilde M_t=-\infty\right\}=0,$$
which implies that 
$ \P\left(\widetilde{L}(0,\infty)\wedge\widetilde{L}(1,\infty)<\infty \right)=0,$
which in turn concludes the proof.
\end{proof}

\subsection{A non-convergence theorem: a general result.}

In this subsection, we prove a non-convergence theorem applicable to a general class of c\`adl\`ag finite variation processes. This is inspired by  techniques developed by Bena\"im and Raimond for continuous-time martingales in \cite{BR2005} and by Limic and Tarr\`es for discrete-time martingales in \cite{LT2007, LT2008}. Some of the notation used in this Section overlaps with the one from other Sections. This will hint on how we are going to apply these general results to VRJP.
\begin{definition}\label{def:good}
Consider a process  $\mathbf{Z} = (Z_t)_{t\ge0}$ and denote by $(\Fcal_t)_{t\ge0}$ its natural filtration. The process $\mathbf{Z}$ is {\bf good} if it is a c\`adl\`ag, with finite variation and can be decomposed as 
$$Z_t=Z_0+ \int_0^t F_{u-} \rmd A_u +\int_0^t G_{u-} \rmd M_u,$$
where $(F_t)_{t\ge0}, (G_t)_{t\ge0}, (A_t)_{t\ge0}, (M_t)_{t\ge0}$ are $(\mathcal{F}_t)_{t\ge0}$-adapted c\`adl\`ag finite variation stochastic processes on $\mathbb{R}$ and satisfy the following.
\begin{itemize} 
\item[(i)] $(M_t)_{t\ge0}$ is a martingale w.r.t. $(\mathcal{F}_t)_{t\ge0}$ such that $\langle M\rangle_t=\int_0^t \Lambda_u \rmd u$ where $(\Lambda_t)_{t\ge0}$ is a positive $(\mathcal{F}_t)_{t\ge0}$-adapted process.
\item[(ii)] $(A_t)_{t\ge0}$ can be decomposed as  $A_t=A_t^+-A_t^-$, for all $t\ge 0$, where $(A^+_t)_{t\ge0}$ and $(A^{-}_t)_{t\ge0}$ are $(\mathcal{F}_t)_{t\ge0}$-adapted and increasing in $t$.  Denote by $({\rm Va}_t(A))_{t\ge0}$ the variation process of $(A_t)_{t\ge0}$ which is given by ${\rm Va}_t(A):= A_t^+ +A_t^-$.
\item[(iii)] For each $t>0$, $\Delta A_t \Delta M_t =0$ almost surely.
\item[(iv)] There is a deterministic constant $K$ such that $$\P\left( \max\left\{|Z_t|,\int_0^t G_u^2\Lambda_u \rmd u \right\}\le K \text{ for all } t\ge0\right)=1.$$
 \end{itemize}
\end{definition}

In what follows, $\varrho$ is a fixed arbitrary real number and $(\Gamma_t)_{t\ge0}$ is a family of events such that for each $t\ge0$, $\Gamma_t\in \mathcal{F}_t$ and 
\begin{equation}\label{intersection}\bigcap_{s\ge t}\Gamma_s=\bigcap_{s\in[t,\infty)\cap \mathbb Q}\Gamma_s.\end{equation}
Furthermore, we assume that
\begin{equation}\label{def.Gamma}\Gamma:=\bigcup_{n=1}^{\infty}\bigcap_{t\ge n}\Gamma_t\subseteq \left\{\lim_{t\to\infty}Z_t=\varrho\right\}.\end{equation}
We are going to apply later on Theorem~\ref{nonconvergence} below to different choices of $(\Gamma_t)_{t\ge0}$ related to strongly VRJP on $\{0,1\}$.
\begin{theorem} \label{nonconvergence}  Suppose that $\mathbf{Z}$ is good and there exist two positive non-increasing $(\mathcal{F}_t)_{t\ge0}$-adapted processes $(\alpha_t)_{t\ge0}$, $(\beta_t)_{t\ge0}$ such that
\begin{equation}\alpha_t=\int_t^{\infty} \widetilde{\alpha}_u\rmd u \quad \text{for some}\ (\mathcal{F}_t)_{t\ge0}\text{-adapted process}\ (\widetilde{\alpha}_t)_{t\ge0},\end{equation} 
\begin{equation}\label{jumpsize}
\max\left\{|\Delta Z _t|, \int_{t}^{\infty} | F_{u-}|  {
 \rm d}{\rm Va}_u(A) \right\}  \le \beta_t
 \end{equation} 
for all $t\ge0$ and
\begin{equation}\label{lim.alpbeta}
\lim_{t \ti} \frac{\beta_t}{\sqrt{\alpha_t}}= 0.
\end{equation}
Furthermore, assume that there exists a deterministic constant $\kappa\in[1,\infty)$ such that for each $t\ge0$,   
\begin{equation}
\label{quadvar} \frac{1}{\kappa} \widetilde{\alpha}_t \le  G^2_{t} \Lambda_{t}  \le \kappa \widetilde{\alpha}_t\quad \text{on the event $\Gamma_t$.}
\end{equation} 
Then   $$\P\left(\Gamma\right)=0.$$
\end{theorem}

{Roughly speaking, the theorem states that the event $\Gamma$ cannot happen if the tail of the finite variation integral term decreases to $0$ faster than the tail of the martingale integral term. We are interested in the case when $\Gamma$ is an event which implies (or is equivalent to) the convergence of the process $\mathbf Z$ to a random variable having atoms. This type of non-convergence theorem has considerable application to stochastic approximation algorithms (see \cite{B1999}) and random processes with reinforcement (see for instance \cite{B1997, BR2005, LT2007, LT2008, OM2018}).}

For any fixed $\varrho$, we can consider the process $(Z_t - \varrho)_t$. This process is good if $\mathbf{Z}$ is. Hence, it is enough to prove the Theorem for $\varrho =0$. Moreover,   we are going to prove Theorem~\ref{nonconvergence} replacing  $\Gamma_t$ with $\widetilde{\Gamma}_t=\Gamma_t\cap\{\sqrt{\alpha_t}\ge 2\beta_t\}$. 

In fact, if we prove Theorem~\ref{nonconvergence} with $(\widetilde{\Gamma}_t)_t$, then we would prove that 
$$\P\left(\bigcup_{n=1}^{\infty}\bigcap_{t\ge n}\widetilde{\Gamma}_t\right)=0.$$ 
On the other hand, 
$$\bigcup_{n=1}^{\infty}\bigcap_{t\ge n}\widetilde{\Gamma}_t=\Gamma\cap\{\exists T: \sqrt{\alpha_t}\ge 2\beta_t\ \forall t\ge T\}.$$
Hence,  using \eqref{lim.alpbeta}, we must have $\P(\Gamma)=0$. 

In order to prove Theorem~\ref{nonconvergence} we need some preliminary results. To simplify the notation, we drop the tilde and use $\Gamma$.
To summarise, from now on we work under the assumption $\varrho=0$ and assume that 
\begin{equation}\label{alpbeta}\sqrt{\alpha_t}\ge 2\beta_t,
\end{equation}
holds on $\Gamma_t$.

For each $t\ge0$, we define the following stopping times 
\begin{align}
\label{defn:St}S_t&=\inf\left\{s\ge t : |Z_s|> \sqrt{8\kappa\alpha_s} \right\},\\
\label{defn:Ut}
U_t & =\inf\left\{s\in [t, \infty): \1_{\Gamma_s}=0\right\}.
\end{align}
{Using \eqref{intersection}, we notice that $U_t$ is well-defined and measurable}. On the other hand, {for any $s\ge t$ we have 
\begin{equation}\label{eq:inclu} \{s<U_t\}\cap \{S_t<U_t\}  \subseteq \{S_t<\infty\}\cap \Gamma_s.
\end{equation}
}
{We first demonstrate that given $\mathcal{F}_t$, the {probability that} there exists a time $s\in[t,\infty)$ when $\Gamma_s$ does not occur or $Z_s\notin[-\sqrt{8\kappa \alpha_s},\sqrt{8\kappa \alpha_s}]$ is uniformly larger than 0. Note that by  \eqref{quadvar}, on the event $\Gamma_t$, the quantity $\alpha_t$ is comparable to the angle bracket of the tail of the martingale integral term.}

\begin{proposition}\label{lem1} Under the assumptions of Theorem~\ref{nonconvergence}, there exists a $(\mathcal{F}_t)_{t\ge0}$-adapted process $(\gamma_t)_{t\ge0}$ such that $\lim_{t\to\infty} \gamma_t=\gamma_{\infty}>0$ a.s. and
$$\mathbb{P}\left(S_t\wedge U_t<\infty|\mathcal{F}_t\right)\ge \gamma_t$$ for all $t>0$.
\end{proposition}
\begin{proof}
Set $\mathcal{A}_t:= \{S_t\wedge U_t=\infty\}$. Using \eqref{def.Gamma}, we have $\{U_t = \infty\}\subseteq \{\lim_{t \to \infty} Z_t = 0\}$.   On the other hand,  on the event $\mathcal{A}_t^c$,  using the inequality $(a+b)^2\le 2a^2 + 2b^2$, we have
\begin{align}\label{Zupper}Z_{S_t \wedge U_t}^2 \le {2 (Z_{S_t \wedge U_t}-Z_{(S_t \wedge U_t)-})^2+2Z_{(S_t \wedge U_t)-}^2} \stackrel{(\ref{jumpsize})+(\ref{defn:St})}{\le}  2\beta_{S_t \wedge U_t}^2+16\kappa\alpha_{S_t\wedge U_t}
 \le 2(\beta_t^2+8\kappa\alpha_t),\end{align}
where the second inequality follows from (\ref{jumpsize}) and (\ref{defn:St}). Therefore,
 \begin{equation}
\begin{aligned}\label{upperb} \mathbb{E}\left[Z_{S_t\wedge U_t}^2 | \mathcal{F}_t\right] & = \mathbb{E}\left[ Z_{S_t\wedge U_t}^2 \1_{\mathcal{A}_t^c} | \mathcal{F}_t\right] \le 2(\beta_t^2+8\kappa\alpha_t) \P\left(\mathcal{A}_t^c \big| \mathcal{F}_t \right).
\end{aligned}
\end{equation}
Applying the integration by parts formula  \eqref{changvar} to $(Z_t)_{t\ge0}$ and using and the fact that $\Delta A_u \Delta M_u=0$ almost surely, we obtain that
for $t\ge s$,
\begin{equation*}
\begin{aligned}
Z_t^2-Z_s^2 &= 2\int_{s}^t Z_{u-}\text{d}Z_u+[Z]_t-[Z]_s\\
 & = 2 \int_{s}^t Z_{u-}F_{u-} \text{d}A_u + 2\int_s^t  Z_{u-}G_{u-} \text{d}M_u 
  + \sum_{s<u\le t} F_{u-}^2 (\Delta A_u)^2 +\sum_{s<u\le t} G_{u-}^2 (\Delta M_u)^2.
\end{aligned}
\end{equation*}
Using \eqref{isometry}-\eqref{M.integrak} and the fact that $G_u=G_{u-}$ almost everywhere, we have \begin{align*}
 \E\left[\sum_{s<u\le t} G_{u-}^2(\Delta M_u)^2  \ \big|\ \mathcal{F}_s \right]  
 & = \E\left[\left[ \int_{0}^{\bullet} {G_{u-}\text{d}M_u}\right]_{t}-\left[ \int_{0}^{\bullet} {G_{u-}\text{d}M_u}\right]_{s}
 \big| \mathcal{F}_{s} \right]\\
 & \stackrel{\eqref{isometry}}{=} \E\left[\left\langle \int_{0}^{\bullet} {G_{u-}\text{d}M_u}\right\rangle_{t}-\left\langle \int_{0}^{\bullet} {G_{u-}\text{d}M_u}\right\rangle_{s}
 \big| \mathcal{F}_{s} \right]
 \\ & \stackrel{\eqref{M.integrak}}{=}\E\left[\int_s^t G_{u-}^2\rmd \langle M\rangle_u  \ \big|\ \mathcal{F}_s\right] =\E\left[\int_s^t G_u^2\Lambda_u\rmd u \ \big|\ \mathcal{F}_s \right].
\end{align*}
{Here the bullet notation $\bullet$ is used to indicate that we consider the whole processes  under angle bracket and quadratic variation operators}. Also note that $\int_{0}^tZ_{u-}G_{u-}\rmd M_u$ is a $L^2$-martingale as its angle bracket 
$$\int_{0}^tZ^2_{u}G^2 _{u} \Lambda_u \rmd u \stackrel{\text{\tiny Def.}~\ref{def:good}(iv)}{\le} K^3    \qquad \mbox{for all $t\ge 0$}.$$
We thus have
\begin{align}\label{2rdmoment}
\E\left[Z_t^2 |\mathcal{F}_s\right]=Z_s^2+  \E\left[ 2\int_{s}^t Z_{u-}F_{u-}\text{d}A_u+  \sum_{s<u\le t} F_{u-}^2 (\Delta A_u)^2 \ \big|\ \mathcal{F}_s\right]+ \E\left[ \int_{s}^t G^2_{u} \Lambda_u\text{d}u \ \big|\ \mathcal{F}_s \right].
\end{align}
Using (\ref{quadvar}), we get
\begin{align*}\int_{t}^{S_t\wedge U_t} {\Lambda_u}G_{u}^2\text{d}u &\ge \1_{\mathcal A_t} \int_{t}^{\infty}{\Lambda_u}G_{u}^2\text{d}u \stackrel{(\ref{quadvar})}{\ge} \frac 1 \kappa \1_{\mathcal A_t} \int_{t}^{\infty}\widetilde{\alpha}_u\text{d}u  = \frac1{\kappa} \1_{\mathcal A_t}  \alpha_t. 
\end{align*}
On the other hand, it follows from \eqref{jumpsize} and \eqref{defn:St} that
\begin{align*}
& \int_{t}^{S_t\wedge U_t} Z_{u-} F_{u-}\text{d}A_u {\ge} - \int_{t}^{S_t\wedge U_t} | Z_{u-}|   | F_{u-}|  \text{d}\text{Va}(A)_u \stackrel{\eqref{defn:St}}{\ge}- \sqrt{8\kappa\alpha_t}\int_{t}^{\infty} | F_{u-}|  \text{d}\text{Va}(A)_u \stackrel{(\ref{jumpsize})}{\ge} -\sqrt{8\kappa\alpha_t}\beta_t.
\end{align*}
Taking into account (\ref{2rdmoment}) and the above inequalities, we have
\begin{align}\label{lowerb}\nonumber \mathbb{E}\left[ Z_{S_t\wedge U_t}^2 | \mathcal{F}_t\right]&\ge 2\E \left[\int_{t}^{S_t\wedge U_t} Z_{u-} F_{u-}\text{d}A_u \big|\mathcal{F}_t\right] + \E\left[ \int_t^{S_t\wedge U_t}\Lambda_u G_{u}^2\text{d}u\ \big|\ \mathcal{F}_t\right]\\
 & \ge  -4\sqrt{2\kappa\alpha_t}\beta_t + \frac{1}{\kappa}\alpha_t\P(\mathcal A_t\big| \mathcal{F}_t).
\end{align}
Combining (\ref{upperb}) with (\ref{lowerb}),  we obtain
\begin{align}\label{probA} \P(\mathcal A_t\big| \mathcal{F}_t)\le\frac{4\sqrt{2\kappa\alpha_t}\beta_t+2(\beta_t+8\kappa\alpha_t)}{\frac{1}{\kappa}\alpha_t+2(\beta_t+8\kappa\alpha_t)}
=1-\gamma_t,
\end{align}
where $$\gamma_t:=\frac{\frac{1}{\kappa}-4\sqrt{2\kappa}\frac{\beta_t}{\sqrt{\alpha_t}}}{\frac{1}{\kappa}+16\kappa+2\frac{\beta_t}{\alpha_t}}\to\frac{\frac{1}{\kappa}}{\frac{1}{\kappa}+16\kappa}>0\quad \text{as} \quad t\to\infty.$$
\end{proof}

{Recall that $S_t$ is the infimum time $s\in[t,\infty)$  such that $Z_s\notin[-\sqrt{8\kappa \alpha_s},\sqrt{8\kappa \alpha_s}]$. We next prove that given $\mathcal{F}_{S_t}$ the probability that $\Gamma$ does not occur or there exists a time $s\in[t,\infty)$ when $\Gamma_s$ does not occur is larger than $1/2$.}

\begin{proposition}\label{lem2} Under the assumptions of Theorem~\ref{nonconvergence}, for all $t\ge 0$, we have that on the event $\{{S_t}<U_t\}$
$$\P\left(\Gamma^c\cup\{U_t<\infty\} |\mathcal{F}_{S_t}\right) \ge \frac{1}{2}.$$
\end{proposition}
\begin{proof}
Set 
\begin{equation}\label{def.Vt}V_t=\inf\left\{s\ge S_t: Z_s=0\right\}.\end{equation}
On the event $\{S_t<U_t\}$,
we have that for each $s\in[ S_t,U_t\wedge V_t]$,
\begin{align}\nonumber
|Z_s|=\left|Z_{S_t}-\int_{S_t}^{s} F_{u-}\text{d}A_u -\int_{S_t}^s G_{u-} \text{d}M_u\right| 
& \ge \left|Z_{S_t} \right| - \int_{S_t}^{\infty} |F_{u-}|\text{d}\text{Va}(A)_u -\left| \int_{S_t}^s G_{u-} \text{d}M_u\right|
\\ \label{Zsupper}
& \stackrel{(\ref{defn:St})+(\ref{jumpsize})}{\ge} \sqrt{8\kappa\alpha_{S_t}} -\beta_{S_t}- \sup_{s\in[S_t, U_t]}\left|\int_{S_t}^s G_{u-}{\text{d}M_u}\right|.
\end{align}
Notice that using \eqref{eq:inclu} with $s= S_t$, we have that $\{S_t<U_t\}\subseteq \Gamma_{S_t}$. Using this fact 
and \eqref{alpbeta},  it follows from \eqref{Zsupper} that on the event $$\{S_t<U_t\} \cap\left\{\sup_{s\in[S_t, U_t]}\left|\int_{S_t}^s G_{u-}{\text{d}M_u}\right|\le\sqrt{2\kappa\alpha_{S_t}} \right\}$$
we have 
\begin{equation}\label{eq:inbe} \inf_{s\in [S_t, U_t\wedge V_t]}|Z_s| \stackrel{\eqref{Zsupper}}{\ge} {\sqrt{2\kappa\alpha_{S_t}}}-\beta_{S_t}\stackrel{(\ref{alpbeta})+ (\kappa \ge 1)}{\ge}\beta_{S_t}.
\end{equation}
On the other hand,  on the event $\{S_t<U_t\}\in \mathcal{F}_{S_t}$, we have
\begin{align}\label{eq:int2}\int_{S_t}^{U_t}{\Lambda_u}G_{u}^2\text{d}u\stackrel{\eqref{quadvar}}{\le} \kappa \int_{S_t}^{U_t}\widetilde{\alpha}_u\text{d}u\le \kappa  \alpha_{S_t}.  
\end{align}
Thus, on the event $\{S_t<U_t\}\in \mathcal{F}_{S_t}$, we have
\begin{equation}\label{eqn:int1}
\begin{aligned}
\P\left(\inf_{s\in [S_t, U_t\wedge V_t]}|Z_s|<\beta_{S_t}  \big|\mathcal{F}_{S_t}\right)&\stackrel{(\ref{eq:inbe})}{\le} \P\left(\sup_{s\in[S_t, U_t]}\left|\int_{S_t}^s G_{u-}{\text{d}M_u}\right|>\sqrt{2\kappa\alpha_{S_t}} \big|\mathcal{F}_{S_t}\right)\\
&\stackrel{\mbox{\tiny Doob ineq.}}{\le}  \frac{\E\left[\left( \int_{S_t}^{U_t}G_{u-}{\text{d}M_u}\right)^2\big|\mathcal{F}_{S_t} \right]}{2\kappa\alpha_{S_t}}
 \\ 
& = \frac{\E\left[ \int_{S_t}^{U_t}\Lambda_u G_{u}^2\text{d}u\big|\mathcal{F}_{S_t} \right]}{2\kappa\alpha_{S_t}}
 \stackrel{(\ref{eq:int2})}{\le}  \frac{1}{2}. 
 \end{aligned}
 \end{equation}
Notice that on the event $\{S_t<U_t\}$,
\begin{align}\P\left(\{U_t=\infty\}\cap\Gamma\big| \mathcal{F}_{S_t} \right)
\nonumber & \stackrel{\eqref{def.Gamma}}{\le} \P\left(\{U_t=\infty\}\cap\{\lim_{s\to\infty} Z_s=0\}\big| \mathcal{F}_{S_t} \right)\\ 
\nonumber & \stackrel{\eqref{def.Vt}}{\le} \P\left(\{U_t=\infty\}\cap\left\{\inf_{s\in [S_t, U_t\wedge V_t]}|Z_s|<\beta_{S_t}\right\}  \big|\mathcal{F}_{S_t}\right) \\
\label{eq:int3} & \le \P\left(\inf_{s\in [S_t, U_t\wedge V_t]}|Z_s|<\beta_{S_t}  \big|\mathcal{F}_{S_t}\right), 
\end{align}
{where the second inequality follows from the definition of $V_t$ in \eqref{def.Vt} and the fact that on the event $\{S_t<U_{t}=\infty\}\cap \{\lim_{s\to\infty} Z_s=0\}$ we have $\inf_{s\in [S_t, U_t\wedge V_t]}|Z_s|=\inf_{s\in [S_t, V_t]}|Z_s|=0<\beta_{S_t}.$}

Hence, on the event $\{S_t<U_t\}$, we obtain  
\begin{align*}\P(\{U_t=\infty\}\cap \Gamma|\mathcal{F}_{S_t} )
\stackrel{\eqref{eqn:int1}+\eqref{eq:int3}}{\le} 
 \frac{1}{2}.
\end{align*}
\end{proof}

\begin{proof}[Proof of Theorem~\ref{nonconvergence}]
Using Proposition~\ref{lem2}, we have
\begin{align*}
\P\left(\Gamma^c\cup\{U_t<\infty\}  \big| \mathcal{F}_t \right) &\ge  \E \left[ \P\left(\Gamma^c \cup\{U_t<\infty\} \big| \mathcal{F}_{S_t}\right)\1_{\{S_t<U_t\}}\big| \mathcal{F}_t\right]\ge \frac{1}{2}\P\left( S_t<U_t \big| \mathcal{F}_t \right) \quad \forall t\ge 0.
 \end{align*}
Other other hand, using Proposition~\ref{lem1}
$$\gamma_t \le \P\left( S_t\wedge U_t<\infty \big| \mathcal{F}_t \right)\le \P\left( S_t<U_t \big| \mathcal{F}_t \right) + \P\left( U_t<\infty \big| \mathcal{F}_t \right).$$ 
Hence it follows that
\begin{align}\label{PGammac}
\P\left(\Gamma^c \cup\{U_t<\infty\} \big| \mathcal{F}_t \right)\ge \frac{1}{2}\left(\gamma_t-\P(U_t<\infty \big| \mathcal{F}_t)\right).
\end{align}
Note that, by the definition of $U_t$ given in \eqref{defn:Ut}, we have that  $\Gamma=\{U_t=\infty, \mbox{for all large $t$}\}$ and thus $\1_{\{U_{t}=\infty\}} \to \1_{\Gamma}$ a.s. as $t\to\infty$. Using the triangle inequality, we have that
\begin{align*}
\E\left[\left|\1_{\Gamma}-\P(U_{t}=\infty|\mathcal{F}_t)\right|\right]
& \le \E\left[\left|\1_{\Gamma}-\P(\Gamma|\mathcal{F}_t)\right|\right]+\E\left[\left|\P(\Gamma|\mathcal{F}_t)-\P(U_{t}=\infty|\mathcal{F}_t)\right|\right] \\
& \le \E\left[\left|\1_{\Gamma}-\P(\Gamma|\mathcal{F}_t)\right|\right]+\E\left[\left|\1_{\Gamma}-\1_{\{U_{t}=\infty\}}\right|\right]
\end{align*}
{By Levy's zero-one law, we note that a.s. $\P(\Gamma|\mathcal{F}_t)\to \1_{\Gamma}$ as $t\to\infty$. Additionally, by Lebesgue's dominated convergence theorem, we also have $\E\left[\left|\1_{\Gamma}-\1_{\{U_{t}=\infty\}}\right|\right]\to 0$ as $t\to\infty$.} As a result, we obtain
\begin{align}\label{Pcontra}\P\left( U_t<\infty \big| \mathcal{F}_t \right) \to \1_{\Gamma^c}\end{align} in $L^1(\P)$ as $t\to\infty$. {Similarly, we also have
\begin{align}\label{Pcontra2}\P(\Gamma^c\cup \{U_t<\infty\}|\mathcal{F}_t)\to \1_{\Gamma^c}\end{align}
 in $L^1(\P)$ as $t\to\infty$.} Combining \eqref{PGammac} with \eqref{Pcontra} {and \eqref{Pcontra2}}, we thus get
$$\1_{\Gamma^c}\ge \frac{1}{2}(\gamma_{\infty}-\1_{\Gamma^c})\quad \text{a.s.}$$
{Since $\gamma_{\infty}>0$ (by Proposition~\ref{lem1}), it follows that $\P(\Gamma)=0$.} 
\end{proof}
\subsection{Strongly VRJP on two vertices}
In this subsection, we consider the extension $\widetilde{\X}$ of the VRJP$(w)$ $\X$ on $\{0,1\}$. Moreover we assume throughout this subsection that the weight function  $w:[0,\infty)\to [0,\infty)$ satisfies Assumption-W, and 
\begin{align}\label{integrable}\int_{1}^\infty\frac{\text{d}t}{w(t)}<\infty.\end{align}
{It immediately  follows that $w(\infty)=\infty$.
We aim to prove that $\widetilde{L}(0,\infty)\wedge \widetilde{L}(1,\infty)<\infty$ a.s. and the distribution of this random variable has no atoms.}

Define \begin{equation}\label{def:Zt}
Z_t:= \int_{1}^{\widetilde{L}(0,t)}\frac{1}{w(u)}\text{d}u-\int_{1}^{\widetilde{L}(1,t)}\frac{1}{w(u)}\text{d}u -\frac{\1_{\{\widetilde X_t=1 \}}}{W(0,t)W(1,t)} = H_t - \frac{\1_{\{\widetilde X_t=1 \}}}{W(0,t)W(1,t)}.\end{equation}
Recall from (\ref{Ht}) that 
\begin{align}\label{def:ZZt}Z_t: =\int_0^t F_{u-}\text{d}A_u+
\int_0^t G_{u-}\text{d}M_u, 
\end{align}
where 
\begin{equation*}F_t:=\1_{\{\widetilde X_{t}=1\}}, 
\quad A_t:=G_t:=-\frac{1}{W(0,t)W(1,t)}\end{equation*} 
and
\begin{equation*}
M_t:=\1_{\{\widetilde X_t=1\}}-\int_{0}^t\left(W(1,u)\1_{\{\widetilde X_u=0\}}-W(0,u)\1_{\{\widetilde X_u=1\}}\right) \text{d}u.
\end{equation*}

{Using \eqref{integrable}, we have that $Z_t$ (defined in \eqref{def:Zt}}) converges almost surely to some $Z_{\infty}$ as $t\to\infty$. On the other hand, it immediately follows from (\ref{def:Zt}) that \begin{equation}\label{Linfty}
\left\{ Z_{\infty}=0\right\}=\left\{  \widetilde{L}(0,\infty)\wedge \widetilde{L}(1,\infty) =\infty\right\}.\end{equation}
Furthermore, the distribution of $\widetilde{L}(0,\infty)\wedge  \widetilde{L}(1,\infty)$ has no atoms if and only if the distribution of $Z_{\infty}$ has no atoms. Hence, in order to prove that $\widetilde{L}(0,\infty)\wedge \widetilde{L}(1,\infty)$ is a finite random variable with non-atomic distribution, it is sufficient to show that 
\begin{equation}\label{eq:aim}\mathbb P\left(Z_{\infty}=\varrho\right)=0.
\end{equation}
for any deterministic $\varrho\in \mathbb R$. 
We need few intermediate results in order to prove \eqref{eq:aim}.

\begin{proposition}
The process $\mathbf{Z}$ is {\bf good}.
\end{proposition}
\begin{proof}
The decomposition \eqref{def:ZZt} fits the one appearing in Definition~\ref{def:good}. Next we check the items in the list appearing in Definition~\ref{def:good}. First,  the condition (i) follows directly from Lemma~\ref{martingale}, where we recall that $(M_t)_{t\ge0}$ is a c\`adl\`ag martingale with $\left\langle M\right\rangle_t=\int_{0}^t \Lambda_u\text{d}u,$
where $$\Lambda_t:=W(0,t) \1_{\{  \widetilde X_t=1\}}+W(1,t) \1_{\{ \widetilde X_t=0\}}.$$ 
The condition (ii) is clear as $A_t$ is increasing for $t\in [0,\infty)$ and thus $A_t=\text{Va}_t(A)$. The condition (iii) is also fulfilled as we showed in the proof of Proposition~\ref{pr:deco} that $\Delta M_t \Delta A_t=0$ a.s. for all $t>0$. It is immediate from \eqref{def:Zt} that  $\sup_{t\ge0}|Z_t|\le 1+2\int_1^{\infty}\frac{du}{w(u)}$. Finally, we verify (iv) by showing that $\int_0^\infty G_u^2 \Lambda_u\rmd u\le2\int_{1}^{\infty}\frac{\rmd u}{w(u)}<\infty$. Indeed, using the fact that $W(0,t)\wedge W(1,t)\ge w(1)=1$ and ${W(0,t)\vee W(1,t)}\ge w(t/2+1)$, we have $$G^2_t\Lambda_t=\frac{\1_{\{\widetilde X_t=1\}}}{W(0,t)W(1,t)^2}+\frac{\1_{\{\widetilde X_t=0\}}}{W(1,t)W(0,t)^2}\le \frac{2}{W(0,t)\vee W(1,t)}\le \frac{1}{w(t/2+1)}.$$
\end{proof}

\begin{proposition}\label{lem:alpha} For any $p>q\ge1$, we have 
$$\frac{1}{\left(W(0,t)W(1,t)\right)^{p/2}}= o\left( \int_{\widetilde{L}(0,t)\wedge \widetilde{L}(1,t)}^{\infty}\frac{\rmd u}{w(u)^q}\right),\quad\text{as}\ t\to\infty.$$
\end{proposition}
\begin{proof}
We first claim that 
\begin{align}\label{limw}
\lim_{t\to\infty}w(t)^p\int_t^{\infty}\frac{\text{d}u}{w(u)^q}=\infty.
\end{align}
Indeed, set $y(t)=\int_t^{\infty}w(u)^{-q}{\text{d}u}$ and note that $$\int_{1}^{\infty}\frac{-y'(t)}{y(t)^{q/p}}\rmd t={\frac{y(1)^{1-q/p}}{1-q/p}}<\infty.$$
Hence we must have $$-\frac{y'(t)}{y(t)^{q/p}}=\left( \frac{1}{w(t)\left(\int_{t}^{\infty} w(u)^{-q} {\rmd u}\right)^{1/p}}\right)^q\to 0$$ as $t\to\infty$ and \eqref{limw} thus immediately follows. We now examine the two following cases:

\noindent{\bf Case 1.}  $\widetilde{L}(0,\infty) \wedge \widetilde{L}(1,\infty)=\infty$, a.s..  Using \eqref{limw}, we get  
$$\frac{1}{(W(0,t)W(1,t))^{p/2}}\le \frac{1}{w(\widetilde{L}(0,t)\wedge \widetilde{L}(1,t))^{p}}=o\left(  \int_{\widetilde{L}(0,t)\wedge \widetilde{L}(1,t)}^{\infty}\frac{\text{d}u}{w(u)^q}\right).$$

\noindent{\bf Case 2.} $\widetilde{L}(0,\infty) \wedge \widetilde{L}(1,\infty)<\infty$. Note that 
$$\int_{\widetilde{L}(0,t) \wedge \widetilde{L}(1,t)}^{\infty}\frac{\text{d}u}{w(u)^q}\ge \int_{\widetilde{L}(0,\infty) \wedge \widetilde{L}(1,\infty)}^{\infty}\frac{\text{d}u}{w(u)^q}>0, \forall t\ge 0$$ while 
$$\lim_{t \ti} \frac{1}{(W(0,t)W(1,t))^{p/2}}= 0.$$ 
This ends the proof  of the lemma. 
\end{proof}

{Using Theorem~\ref{nonconvergence}, we are going to prove that  $W(0,t)$ and $W(1,t)$ cannot both grow with the same order.}
\begin{proposition}\label{lem:liminf}We have that  $$\liminf_{t\to\infty}  \frac{W(0,t)\wedge W(1,t)}{W(0,t)\vee W(1,t)}=0, {\qquad a.s..}$$
\end{proposition} 
\begin{proof}
For each $t\in [0,\infty)$ and $k\in \N$, set 
\begin{equation}\label{eq:defg} \Gamma^{(k)}_t:=\left\{\frac{1}{k}\le \frac{W(1,t)}{W(0,t)}\le k\right\}\quad \text{and}\quad \Gamma^{(k)}:=\left\{ \exists T: \frac{1}{k}\le \frac{W(1,t)}{W(0,t)}\le k, \forall t\ge T \right\}=\bigcup_{n=1}^{\infty}\bigcap_{t\ge n} \Gamma^{(k)}_t.
\end{equation}
We observe that as at least one of the local times must diverge to $\infty$, i.e. $L(0,\infty)\vee L(1,\infty) =\infty$, we have that $\Gamma^{(k)}\subseteq \left\{ \widetilde L(0,\infty)=\widetilde L(1,\infty)=\infty\right\}=\left\{\lim_{t\to\infty} Z_t=0\right\}$. Moreover,
$$\left\{\liminf_{t\to\infty}  \frac{W(0,t)\wedge W(1,t)}{W(0,t)\vee W(1,t)}>0\right\}=\bigcup_{k=1}^{\infty}\Gamma^{(k)}.$$
Hence, we only have to show that for each $k\ge 1$, \begin{align}\label{no.Gk}\P\left(\Gamma^{(k)}\right)=0.
\end{align}
Now let $k$ be a fixed positive integer. For $t\ge0$, we define 
\begin{equation}\widetilde{\alpha}_t:= \frac{\1_{\{\widetilde X_t=0\}}}{W(0,t)^3}+\frac{\1_{\{\widetilde X_t=1\}}}{W(1,t)^3}\quad \text{and}\quad \alpha_t:=\int_t^{\infty}\widetilde{\alpha}_u\rmd u \stackrel{\mbox{\tiny change of var.}}{=}\int_{\widetilde L(0,t)}^{\infty}\frac{\text{d}u}{w(u)^3}+\int_{\widetilde L(1,t)}^{\infty}\frac{\text{d}u}{w(u)^3}.\end{equation}
It is clear that $\alpha_t$ and $\widetilde{\alpha}_t$ are both $\mathcal{F}_t$-measurable, as $w$ is deterministic. We now apply Theorem~\ref{nonconvergence} to the process $(Z_t)_{t\ge0}$ defined as in (\ref{def:ZZt}) and the family of events $(\Gamma^{(k)}_t)_{t\ge0}$.  
Let us now verify assumptions (\ref{jumpsize})-(\ref{lim.alpbeta})-(\ref{quadvar}). For each $t\ge0$, set $$\beta_t:= \frac{1}{W(0,t)W(1,t)}.$$ We notice that
\begin{align*}|\Delta Z_t|=  \frac{| \1_{\{\widetilde X_t=1\}}- \1_{\{\widetilde X_{t-}=1\}}|}{W(0,t)W(1,t)}\le\beta_t,
\end{align*}
and, as $|F_{t}|\le 1$, we have
\begin{align}\label{FVubound}
\int_{t}^{\infty}|F_{u-}|\text{d}\text{Va}_u(A)
\le A_{\infty}-A_t=\beta_t.
\end{align}
The condition \eqref{jumpsize} is thus verified. The condition (\ref{lim.alpbeta}) immediately follows by applying Proposition~\ref{lem:alpha} for $p=4$ and $q=3$. On the other hand, on the event $\Gamma^{(k)}_t$ (defined in \eqref{eq:defg}) 
$$\frac{1}{k}\widetilde\alpha_t\le G^2_{t} \Lambda_{t} =\frac{\1_{\{\widetilde X_t=1\}}}{W(1,t)^2W(0,t)}+ \frac{\1_{\{\widetilde X_t=0\}}}{W(0,t)^2W(1,t)}  \le k \widetilde{\alpha}_t.$$
The condition (\ref{quadvar}) is thus fulfilled. Using Theorem~\ref{nonconvergence}, we obtain \eqref{no.Gk} which concludes the proof.  
\end{proof}
\begin{proposition} \label{martinq}
For $s\ge t\ge0$, we have
\begin{align}\label{Mlower}
\left( \int_t^{s} G_{u-}{{\rm d}M_u}\right)^2 \ge \frac{1}{2}\left(\int_{t}^{s}\left[ \frac{1_{\{\widetilde X_u=0\}}}{W(0,u)}-\frac{1_{\{\widetilde X_u=1\}}}{W(1,u)}\right]{\rm d}u\right)^2-\frac{4}{(W(0,t)W(1,t))^2}
\end{align}
and
\begin{align}\label{Mupper}
\left| \int_t^{s} G_{u-}{{\rm d}M_u} \right| \le 2 \int_{\widetilde{L}(0,t)\wedge \widetilde{L}(1,t)}^{\infty} \frac{{\rm d}u}{w(u)}+\frac{2}{W(0,t)W(1,t)}.
\end{align}
\end{proposition} 
\begin{proof}
Recall from (\ref{Ht}) that
\begin{align*}\int_t^s G_{u-}{{\rm d}M_u}&=\int_{t}^{s}\left[\frac{\1_{\{ \widetilde X_u=0 \}}}{W(0,u)}-\frac{\1_{\{ \widetilde X_u=1\}}}{W(1,u)}\right]\text{d}u\\& - \left( \frac{\1_{\{\widetilde X_s=1 \}}}{W(0,s)W(1,s)} -\frac{\1_{\{\widetilde X_t=1 \}}}{W(0,t)W(1,t)}\right)-\int_t^s \1_{\{\widetilde X_{u-}=1 \}}\text{d}A_u.
\end{align*}
The inequality (\ref{Mlower}) is  obtained by using the inequality $(a-b)^2\ge \frac{1}{2}a^2- b^2$ and the fact that for $s\ge t\ge0$
\begin{align}\label{eq:ineqW} \left| \frac{\1_{\{\widetilde X_s=1 \}}}{W(0,s)W(1,s)} -\frac{\1_{\{\widetilde X_t=1 \}}}{W(0,t)W(1,t)}\right|& \le \frac{1}{W(0,t)W(1,t)},\\
 \label{eq:ineqW2}\left| \int_t^s \1_{\{\widetilde X_{u-}=1 \}}\text{d}A_u\right| & \le \frac{1}{W(0,t)W(1,t)}. \end{align}
On the other hand, inequality  (\ref{Mupper}) is also obtained by combining \eqref{eq:ineqW}-\eqref{eq:ineqW2} and
$$
\int_{t}^{s}\left[\frac{\1_{\{ \widetilde X_u=0 \}}}{W(0,u)}-\frac{\1_{\{ \widetilde X_u=1\}}}{W(1,u)}\right]\text{d}u\le \int_{t}^{\infty}\left[\frac{\1_{\{ \widetilde X_u=0 \}}}{W(0,u)}+\frac{\1_{\{ \widetilde X_u=1\}}}{W(1,u)}\right]\text{d}u= \int_{\widetilde{L}(0,t)}^{\infty}\frac{\text{d}u}{w(u)}+\int_{\widetilde{L}(1,t)}^{\infty}\frac{\text{d}u}{w(u)}.$$
\end{proof}
\begin{proposition}\label{bounded} Assume that $$\int_0^\infty\frac{\rmd u}{w(u)}<\infty$$ and that the function $w$  satisfies the Assumptions of Theorem~\ref{th:main} b),  then for any fixed real  $\varrho$, we have 
$$
\P\left(Z_{\infty}=\varrho\right)=0.
$$
Moreover, the random variable $\widetilde{L}(0,\infty)\wedge \widetilde{L}(1,\infty)$ is almost surely finite and its distribution has no atoms. \end{proposition}

\begin{proof}
{Recall from \eqref{eq:assu} that there exists $\rho>0$ such that $t\mapsto w(t)^{\rho}\int_t^{\infty}\frac{\rmd s}{w(s)}$ is non-increasing. 
It follows that}
$$\frac{\int_{\widetilde{L}(0,t)\vee \widetilde{L}(1,t)}^{\infty} (1/w(u)){\text{d}u} }{\int_{\widetilde{L}(0,t)\wedge \widetilde{L}(1,t)}^{\infty}(1/w(u)){\text{d}u}} \le \left(\frac{W(0,t)\wedge W(1,t)}{W(0,t)\vee W(1,t)}\right)^{\rho}.$$
Hence by virtue of Proposition~\ref{lem:liminf}, we obtain {that a.s.} 
\begin{equation}\label{eq:intozero}
\liminf_{t\to\infty}\frac{\int_{\widetilde{L}(0,t)\vee \widetilde{L}(1,t)}^{\infty}(1/w(u)){\text{d}u}}{\int_{\widetilde{L}(0,t)\wedge \widetilde{L}(1,t)}^{\infty}(1/w(u)){\text{d}u}}=0.
\end{equation}
We define the sequence of stopping times $(\mathfrak{t}_n)_{n\ge0}$ such that $\mathfrak{t}_0=0$ and for all $n\ge1$,
\begin{equation}
\label{def.tn}\mathfrak{t}_n =\inf\left\{{t}\ge \mathfrak{t}_{n-1}+1 \ : \ \int_{\widetilde{L}(0,t)\vee \widetilde{L}(1,t)}^{\infty}\frac{\text{d}u}{w(u)} <\frac{1}{2}\int_{\widetilde{L}(0,t)\wedge \widetilde{L}(1,t)}^{\infty}\frac{\text{d}u}{w(u)}\right\}.
\end{equation}
Notice that using \eqref{eq:intozero}, we infer $\P(\mathfrak{t}_n<\infty)=1$ for each $n\ge1$ and $\mathfrak{t}_n\to\infty$ as $n\to\infty.$

Let $\varrho$ be a fixed real number and set $\widehat{Z}_t={Z}_t-\varrho$ for all $t\ge0$. Define $\mathcal{A}=\{ \widehat{Z}_{\infty}=0\}.$ 
To complete the proof, it is sufficient to show that (for any $\varrho\in \R$)
\begin{equation}\label{non.A}
\P(\mathcal{A})=0.
\end{equation}
Indeed, assuming \eqref{non.A}, the distribution of $Z_{\infty}$ has no atoms.
{Recall from \eqref{def:Zt} and \eqref{Linfty} that $\{\widetilde L(0,\infty)\wedge \widetilde L(1,\infty)=\infty\}=\{Z_{\infty}=0\}$ and the distribution of $\widetilde L(0,\infty)\wedge \widetilde L(1,\infty)$ has no atoms if and only if the distribution of $Z_{\infty}$ has no atoms.} As a consequence, the random variable $\widetilde L(0,\infty)\wedge \widetilde L(1,\infty)$ is finite almost surely and its distribution has no atoms.

The rest of this proof is devoted to prove \eqref{non.A}. For $n\ge1$, set
\begin{equation}\label{def.delta}\delta_{n}:=\left(\int_{\widetilde{L}(0,\mathfrak{t}_n)\wedge \widetilde{L}(1,\mathfrak{t}_n)}^{\infty}\frac{\text{d}u}{w(u)}\right)^2
\end{equation}
and \begin{equation}\label{def.Tn} R_n=\inf\left\{\mathfrak{t}_m:  m\ge n, |\widehat{Z}_{\mathfrak{t}_m}|>8 \sqrt{\delta_{n}}\right \}.
\end{equation}
It is clear that $\delta_{n}$ is $\mathcal{F}_{\mathfrak{t}_n}$-measurable and $R_n$ is a stopping time {w.r.t.} the filtration $(\mathcal{F}_{\mathfrak{t}_n})_{n\ge1}$. 

{Using similar techniques as in the proofs of Proposition~\ref{lem1} and Proposition~\ref{lem2}, we next prove the followings claims.}

\noindent {\bf Claim 1.} There exists a  $(\mathcal{F}_{\mathfrak{t}_n})$-adapted sequence $(b_n)_{n\ge1}$ such that for all $n\ge 1$
\begin{align}
\label{prob1} \P\left(\mathcal{A}\cap \{{R_n}=\infty\}\big| \mathcal{F}_{\mathfrak{t}_n}\right)\le b_n
\end{align}
and $\lim_{n\to\infty} b_n =: b_\infty <1$.\\
\noindent {\bf Claim 2.} There exist a $(\mathcal{F}_{\mathfrak{t}_n})$-adapted sequence  $(c_n)_{n\ge1}$ and a subsequence of stopping times $(\mathfrak{t}_{N_k})_{k\ge1}\subset(\mathfrak{t}_n)_{n\ge1}$ such that for each $k\ge1$, $\mathfrak{t}_{N_k}$ is finite, $\lim_{k\to\infty}\mathfrak{t}_{N_k}\to\infty$ a.s. and
\begin{align}\label{prob2}
\P\left(\mathcal{A}^c\big|\mathcal{F}_{R_{N_k}}\right)\ge c_{N_k} \quad \text{on the event}\ \{R_{N_k}<\infty\} 
\end{align}
and 
$\lim_{n\to\infty}c_n=:c_{\infty}> 0$.

Let us first demonstrate Claim 1.  By virtue of \eqref{isometry}-\eqref{M.integrak} and \eqref{Mlower}, we have
\begin{equation}\label{Mpart}
\begin{aligned}
& \E\left[\int_{\mathfrak{t}_n}^{R_n} {G_u^2\Lambda_u\text{d}u}\big| \mathcal{F}_{\mathfrak{t}_n} \right]  =\E\left[ \int_{\mathfrak{t}_n}^{R_n} G_{u-}^2\text{d}\left\langle M\right\rangle_u\big| \mathcal{F}_{\mathfrak{t}_n} \right]\\
&  \stackrel{\eqref{M.integrak}}{=}\E\left[\left\langle \int_{0}^{\bullet} {G_{u-}\text{d}M_u}\right\rangle_{R_n}-\left\langle \int_{0}^{\bullet} {G_{u-}\text{d}M_u}\right\rangle_{\mathfrak{t}_n}
 \big| \mathcal{F}_{\mathfrak{t}_n} \right] \stackrel{\eqref{isometry}}{ =}\E\left[\left( \int_{\mathfrak{t}_n}^{R_n} G_{u-}\text{d}M_u\right)^2\big| \mathcal{F}_{\mathfrak{t}_n} \right]\\
 & \stackrel{\eqref{Mlower}}{\ge}   \frac{1}{2}\E\left[\left(\int_{\mathfrak{t}_n}^{R_n}\left[ \frac{1_{\{\widetilde X_u=0\}}}{W(0,u)}-\frac{1_{\{\widetilde X_u=1\}}}{W(1,u)}\right]\text{d} u\right)^2 \big| \mathcal{F}_{\mathfrak{t}_n}\right]-\frac{4}{\left(W(0,{\mathfrak{t}_n})W(1,{\mathfrak{t}_n})\right)^2}\\
 & \ge \frac{1}{2}\E\left[ \left(\int_{\widetilde{L}(0,\mathfrak{t}_n)}^{\infty}\frac{\text{d}u}{w(u)}-\int_{\widetilde{L}(1,\mathfrak{t}_n)}^{\infty}\frac{\text{d} u}{w(u)} \right)^2\1_{ \mathcal{A}\cap \{R_n=\infty\}} \big|  \mathcal{F}_{\mathfrak{t}_n}\right]- \frac{4}{\left(W(0,{\mathfrak{t}_n})W(1,{\mathfrak{t}_n})\right)^2}\\
& \stackrel{\eqref{def.tn}+\eqref{def.delta}}{\ge}\frac{\delta_{n}}{8}\P\left(\mathcal{A}\cap \{R_n=\infty\}\big|  \mathcal{F}_{\mathfrak{t}_n}\right)- \frac{4}{\left(W(0,{\mathfrak{t}_n})W(1,{\mathfrak{t}_n})\right)^2}.
\end{aligned}
\end{equation}
On the other hand  for all $u\in [\mathfrak{t}_n,R_n)$, when the latter interval is non-empty, we have 
\begin{align}\label{Zupperbound}
\nonumber|\widehat  Z_{u}|&\stackrel{\eqref{def:ZZt}}{\le} |\widehat Z_{\mathfrak{t}_n}|+\int_{\mathfrak{t}_n}^{\infty} |F_{s-}|\rmd \text{Va}(A)_s+\left|\int_{\mathfrak{t}_n}^{u} G_{s-}\rmd M_s\right|\\
\nonumber
&\stackrel{\eqref{def.Tn}+\eqref{FVubound}+\eqref{Mupper} +\eqref{def.delta}}{\le}  8\sqrt{\delta_{n}}+ \frac{1}{W(0,\mathfrak{t}_n)W(1,\mathfrak{t}_n)} + 2\sqrt{\delta_{n}}+ \frac{2}{W(0,\mathfrak{t}_n)W(1,\mathfrak{t}_n)}\\
& = 10\sqrt{\delta_{n}}+ \frac{3}{W(0,\mathfrak{t}_n)W(1,\mathfrak{t}_n)}.
\end{align}
Hence,
\begin{align}\label{FVpart}
\nonumber \int_{\mathfrak{t}_n}^{R_n} \widehat Z_{u-} F_{u-}\text{d}A_u & \ge - \sup_{\mathfrak{t}_n\le u\le R_n}| \widehat Z_{u}|  \int_{\mathfrak{t}_n}^{\infty}   | F_{u-}|  \text{d}\text{Va}(A)_u \\
& \stackrel{\eqref{FVubound}+\eqref{Zupperbound}}{\ge} -\frac{10 \sqrt{\delta_{n}}}{W(0,{\mathfrak{t}_n})W(1,{\mathfrak{t}_n})} - \frac{3}{\left(W(0,{\mathfrak{t}_n})W(1,{\mathfrak{t}_n})\right)^2}.
\end{align}
Similarly to \eqref{2rdmoment}, we also have
\begin{align}
\label{2rdmoment2}
\E\left[\widehat{Z}_{R_n}^2 |\mathcal{F}_{\mathfrak{t}_n}\right]=\widehat Z_{\mathfrak{t}_n}^2+  \E\left[ 2\int_{\mathfrak{t}_n}^{R_n} \widehat Z_{u-}F_{u-}\text{d}A_u+  \sum_{\mathfrak{t}_n<u\le R_n} F_{u-}^2 (\Delta A_u)^2 \ \big|\ \mathcal{F}_{\mathfrak{t}_n}\right]+ \E\left[ \int_{\mathfrak{t}_n}^{R_n} G^2_{u} \Lambda_u\text{d}u \ \big|\ \mathcal{F}_{\mathfrak{t}_n} \right].
\end{align}
Combining \eqref{Mpart} and \eqref{FVpart} with \eqref{2rdmoment2}, it follows that
\begin{align}\label{eq:comb1} \mathbb{E}\left[ \widehat Z_{R_n}^2 | \mathcal{F}_{\mathfrak{t}_n}\right]&\ge 
-\frac{20 \sqrt{\delta_{n}}}{W(0,{\mathfrak{t}_n})W(1,{\mathfrak{t}_n})} - \frac{10}{\left(W(0,{\mathfrak{t}_n})W(1,{\mathfrak{t}_n})\right)^2}+ \frac{\delta_{n}}{8}\P\left( \mathcal{A}\cap \{R_n=\infty\}\big|  \mathcal{F}_{\mathfrak{t}_n}\right).\end{align}
Using \eqref{Zupperbound}, we also get
\begin{align}\label{eq:comb2} \mathbb{E}\left[\widehat Z_{R_n}^2 | \mathcal{F}_{\mathfrak{t}_n}\right]  & = \mathbb{E}\left[ \widehat Z_{R_n}^2 \1_{\mathcal{A}^c\cup \{{R_n}<\infty\}} | \mathcal{F}_{\mathfrak{t}_n}\right]  +  \mathbb{E}\left[ \widehat Z_{R_n}^2 \1_{\mathcal{A}\cap \{{R_n}=\infty\}} | \mathcal{F}_{\mathfrak{t}_n}\right] \\
\nonumber & = \mathbb{E}\left[ \widehat Z_{R_n}^2 \1_{\mathcal{A}^c\cup \{{R_n}<\infty\}} | \mathcal{F}_{\mathfrak{t}_n}\right]  \\
\nonumber & \le 2\left(\frac{9}{\left(W(0,{\mathfrak{t}_n})W(1,{\mathfrak{t}_n})\right)^2}+100\delta_{n}\right) \P\left(\mathcal{A}^c\cup \{{R_n}<\infty\} \big| \mathcal{F}_{\mathfrak{t}_n} \right).
\end{align}
Therefore, for each $n\ge 1$,
\begin{align*} \P\left(\mathcal{A}\cap \{{R_n}=\infty\}\big| \mathcal{F}_{\mathfrak{t}_n}\right)
&\stackrel{\eqref{eq:comb1}+\eqref{eq:comb2}}{\le}\frac{\frac{20 \sqrt{\delta_{n}}}{W(0,{\mathfrak{t}_n})W(1,{\mathfrak{t}_n})}+\frac{10}{\left(W(0,{\mathfrak{t}_n})W(1,{\mathfrak{t}_n})\right)^2}+2\left(\frac{9}{\left(W(0,{\mathfrak{t}_n})W(1,{\mathfrak{t}_n})\right)^2}+100\delta_{n}\right)}{ \frac{\delta_{n}}{8}+2\left(\frac{9}{\left(W(0,{\mathfrak{t}_n})W(1,{\mathfrak{t}_n})\right)^2}+100\delta_{n}\right)}:= b_n.
\end{align*}
Applying Proposition~\ref{lem:alpha} for $p=2, q=1$, we have \begin{equation}\label{W.delta}\frac{1}{W(0,\mathfrak{t}_n)W(1,\mathfrak{t}_n)}=o(\sqrt{\delta_{n}})
\end{equation} as $n\to\infty$. Thus 
$$\lim_{n\to\infty} b_n =\frac{2\times 100}{2\times 100 +\frac{1}{8}}<1.$$ 
This concludes the proof of Claim 1, and we turn now to the proof of Claim 2. Set $N_0=0$ and we define the sequence $(N_k)_{k\ge1}$ recursively as follows
\begin{equation}\label{W.delta3}
N_{k+1}=\inf\left\{n\ge N_k+1: 2\sqrt{\delta_{n}}\ge \frac{1}{W(0,\mathfrak{t}_{n})W(1,\mathfrak{t}_{n})} \right\}.
\end{equation}
Using \eqref{W.delta}, for each $k\ge 1$, $N_k$ is finite and $\lim_{k\to\infty}N_k= \infty$ a.s.. Furthermore, $\mathfrak{t}_{N_k}$ is a stopping times w.r.t. $(\mathcal{F}_{t})_{t\ge0}$ and $R_{N_k}$ is a stopping times w.r.t $(\mathcal{F}_{\mathfrak{t}_{N_k}})_{k\ge1}$. In what follows, the subscript $k$ is omitted for simplicity.
Notice that on the event $\{R_{N}<\infty\}\cap \left\{\sup_{R_{N}\le s<\infty}\left|\int_{R_{N}}^s{G_{u-}\text{d}M_u}\right|<4\sqrt{\delta_{{N}}} \right\}$, for $s>R_{N}$ we have
\begin{align*}
|Z_s|=\left|Z_{R_{N}}-\int_{R_{N}}^{s} F_{u-}\text{d}A_u -\int_{R_{N}}^s G_{u-} \text{d}M_{N}\right| 
& \ge \left|Z_{R_{N}} \right| - \int_{R_{N}}^{\infty} |F_{u-}|\text{d}\text{Va}(A)_u -\left| \int_{R_{N}}^s G_{u-} \text{d}M_u\right|\\
& \stackrel{\eqref{def.Tn}+\eqref{FVubound}}{\ge} 8\sqrt{\delta_{{N}}}-\frac{1}{W(0,R_{N})W(1,R_{N})}
-4\sqrt{\delta_{{N}}}\\
& \ge 4\sqrt{\delta_{{N}}} - \frac{1}{W(0,\mathfrak{t}_{N})W(1,\mathfrak{t}_{N})}\\
& \stackrel{\eqref{W.delta3}}{\ge} \frac{1}{W(0,\mathfrak{t}_{N})W(1,\mathfrak{t}_{N})}.
\end{align*}
Hence, 
\begin{align*}\{R_{N}<\infty\}\cap \left\{\sup_{R_{N}\le s<\infty}\left|\int_{R_{N}}^s{G_{u-}\text{d}M_u}\right|<4\sqrt{\delta_{{N}}} \right\}\subseteq \left\{\liminf_{s\to\infty}|Z_s|>0\right\}= \mathcal{A}^c\end{align*}
and thus on the event $\{R_{N}<\infty\}$,
\begin{equation}\label{eqn1}\P(\mathcal{A}^c|\mathcal{F}_{R_{N}})\ge \P\left( \sup_{R_{N}\le s<\infty}\left|\int_{R_{N}}^s{G_{u-}\text{d}M_u}\right|<4\sqrt{\delta_{{N}}} \big|\mathcal{F}_{R_{N}} \right).
\end{equation}
On the other hand, the Doob's maximal inequality and (\ref{Mupper}) yield that on event $\{R_{N}<\infty\}$,
\begin{align}
\nonumber\P\left(\sup_{R_{N}\le s<\infty}\left|\int_{R_{N}}^s{G_{u-}\text{d}M_u}\right|>4\sqrt{\delta_{{N}}}\big|\mathcal{F}_{R_{N}} \right)&\le\frac{\E\left[\left(\int_{R_{N}}^{\infty}G_{u-}\text{d}M_u\right)^2\big|\mathcal{F}_{R_{N}}\right]}{16\delta_{{N}}}\\
\label{eqn2} & \le \frac{8\delta_{{N}}+\frac{8}{W(0,R_{N})^2W(1,R_{N})^2} }{16\delta_{{N}}}\le 1-c_{N},
\end{align}
where we set 
$$c_n=\frac12- \frac{1}{2\delta_{n}W(0,\mathfrak{t}_n)^2W(1,\mathfrak{t}_n)^2}.$$ 
Note that $c_{n}$ is $\mathcal{F}_{\mathfrak{t}_{n}}$-measurable and $\lim_{n\to\infty}c_n=\frac{1}{2}.$ Claim 2 thus follows by combining \eqref{eqn1} with \eqref{eqn2}.

We now combine Claims 1 and 2 to finish the proof of the proposition. Taking into account \eqref{prob2}, for $k\ge1$ we have
\begin{align*}
\P\left(\mathcal{A}^c  \big| \mathcal{F}_{\mathfrak{t}_{{N_k}}} \right) &\ge  \E \left[ \P\left(\mathcal{A}^c  \big| \mathcal{F}_{R_{N_k}}\right)\1_{\{R_{N_k}<\infty\}}\big| \mathcal{F}_{\mathfrak{t}_{N_k}}\right]\stackrel{\eqref{prob2}}{\ge} c_{N_k}\P\left( R_{N_k}<\infty \big| \mathcal{F}_{\mathfrak{t}_{N_k}} \right).
 \end{align*}
Hence it follows that
$$\left(\frac{1}{c_{N_k}}+1\right)\P\left(\mathcal{A}^c  \big| \mathcal{F}_{\mathfrak{t}_{N_k}} \right)\ge \P\left( R_{N_k}<\infty \big| \mathcal{F}_{\mathfrak{t}_{N_k}} \right) + \P\left(\mathcal{A}^c  \big| \mathcal{F}_{\mathfrak{t}_{N_k}} \right)\ge  \mathbb{P}(\mathcal{A}^c \cup \{R_{N_k}<\infty\}|\mathcal{F}_{\mathfrak{t}_{N_k}}).$$
As a result of \eqref{prob1}, for $k\ge1$ we have 
$$\P\left(\mathcal{A}^c  \big| \mathcal{F}_{\mathfrak{t}_{N_k}} \right)\ge \frac{1-b_{N_k}}{1+1/c_{N_k}}.$$ 
Using L\'evy's zero-one law, we obtain 
$$1_{\mathcal{A}^c}\ge \frac{1-b_{\infty}}{1+1/c_{\infty}}.$$
Since $b_{\infty}<1$ and $c_{\infty}>0$, we can conclude that $\P(\mathcal{A}^c)=1$, {which proves \eqref{non.A}}. 
\end{proof}

\section{Proof of Theorem \ref{th:main}}
We will make use of the following proposition
\begin{proposition}\label{lem3x}
Let $x\in\Z$. Then, on the event $\{\mathbf{X}=(X_t)_{t\ge0}$ visits $x$ infinitely many times$\}$, we have $\max\{{L}(x-1,\infty),{L}(x,\infty), {L}(x+1,\infty)\}=\infty$.
\end{proposition}
\begin{proof}
 Suppose that ${L}(x-1, \infty)\wedge {L}(x+1,\infty)=:A<\infty$. Every time $\mathbf{X}$ visits $x$, it spends there an exponential amount of time with parameter smaller than  $w(A)$.
 Denote by $T_n$ the time of the $n$-th visit to $x$. We have that ${L}(x,{T}_n)$ is stochastically larger than $$ \frac{1}{w(A)}\sum_{i=1}^n \xi_i$$ where $\xi_i$ are i.i.d.\ exponential random variables with parameter 1. Under the assumption that $x$ is visited infinitely often, all $T_n$ are well defined, and since the above sum a.s.\ diverges, we have ${L}(x,\infty)=\infty$.
\end{proof}
The following result follows from the fact that if $X$ and $Y$ are two independent random variables then $\P(X=Y)=\int \P(X=y)\rmd F_{Y}(y)$, where $F_Y$ is the cumulative distribution function of $Y$.
\begin{lemma}\label{nomatch}
Let $X$ and $Y$ be two independent random variables, and the distribution of at least one of them does not have atoms. Then $\P(X=Y)=0$.
\end{lemma}

\begin{proof}[Proof of Theorem \ref{th:main}. Part (a)]
{By virtue of Theorem~\ref{th:notransient}, $\X$ cannot be transient and thus there exists a.s. a vertex  which is visited infinitely often. Moreover, using Proposition~\ref{lem3x}, we infer that there exists a.s. at least one  vertex with unbounded local time.}
Hence, we only have to show that the process cannot get stuck at any subset of $\mathbb{Z}$. Indeed, setting $T_j=\inf\{s : X_s=j\}$,  we will show that $\P\left(T_j<\infty \right)=1$ for all $j\in \mathbb{Z}$. If for some $j\in\mathbb{Z}$, $\P\left(T_j=\infty \right)>0$ and assuming w.l.o.g that $j>0$, there exists a vertex $k$ such that $\P\left( L(k,\infty)=\infty,  L(k+1,\infty)<\infty \right)>0$. At the same time, applying Proposition~\ref{VRJP2vert} and the restriction principle to $\{ k, k+1\}$ we have $L(k,\infty)=L(k+1,\infty)=\infty$, a.s., which yields a contradiction. {By the same argument using the restriction principle, we have that a.s. $L(k,\infty)=\infty$ for all $k\in\Z$.}  
\end{proof}

\begin{proof}[Part (b)] {We showed in Part (a) that there exists a.s. at least one vertex $x$ such that $L(x,\infty)=\infty$ (note that this fact only requires Assumption W)}. 
The idea of the rest of the proof is the following; set w.l.o.g.\ $x=0$. 
If $L(0,\infty)=\infty$, then by considering the extension of VRJP to $\{0,1\}$ we obtain that vertex $1$ is visited infinitely often but $L(1,\infty)<\infty$. 
Then we must have $L(2,\infty)<\infty$, and here is why. Indeed, if both $0$ and $2$ have their local times going to infinity, then $L(1,\infty)$ must have simultaneously the distribution of $L_\infty$ for two {\it independent} processes, the extensions of VRJP to $\{0,1\}$ and to $\{1,2\}$ respectively.  However, this is impossible by Lemma~\ref{nomatch}. Finally, if $L(2,\infty)<\infty$ as well as $L(1,\infty)<\infty$, then vertex $2$ will be visited only finitely often.
Now we present the formal proof.

Using the restriction principle and Proposition~\ref{bounded}, we get then that  
\begin{equation}\label{3points}\P\left(\bigcup_{x\in \Z} C_x\right)=1,\end{equation}
where we set
$$C_{x}=\left\{
L(x,\infty)=\infty, L(x+1,\infty)<\infty, L( x-1,\infty)<\infty\right\}.$$
We next show that for each $x\in \Z$.
\begin{align}\label{5points}\P\left(\{ L(x+2,\infty)\vee  L(x- 2,\infty)=\infty\} \cap C_x\right)=0.\end{align} 
For simplicity, let us consider the case when $x=0$.  Let $\widetilde{\X}$ and $\widehat{\X}$ be the extensions of VRJP $\X$ to $\{0,1\}$ and $\{1,2\}$ respectively. Denote by $(\widetilde{L}(i,t), i\in\{0,1\})_{t\ge0}$ and $(\widehat{L}(i,t), i\in\{1,2\})_{t\ge0}$ the local times processes of $\widetilde{\X}$ and $\widehat{\X}$ respectively. Notice that on the event $C_0$, $\X$ spends an unbounded amount of time on $\{0,1\}$. Hence, by the restriction principle we must have  $$\widetilde{L}(1,\infty)=L(1,\infty)<\infty \quad\text{ on the event}\ C_0.$$ Furthermore, $\widetilde{L}(0,\infty)\wedge\widetilde{L}(1,\infty)$ has a non-atomic distribution by Proposition~\ref{bounded}. Similarly, on the event $\{L(2,\infty)=\infty\}$, we also have $\widehat L(1,\infty)=L(1,\infty)<\infty$. Since $\widetilde\X$ and $\widehat\X$ are independent,  $\widetilde{L}(0,\infty)$ and $\widetilde{L}(1,\infty)$ do not depend on $\widehat{L}(1,\infty)$. At the same time, $\widetilde{L}(1,\infty)=\widetilde{L}(0,\infty)\wedge\widetilde{L}(1,\infty)$ coincide with  $\widehat{L}(1,\infty)$ on the event $\{L(2,\infty)=\infty\}\cap C_0$. Consequently,
$$
\P\left(L(2,\infty)=\infty\}\cap C_0\right)\le \P\left(\widetilde{L}(0,\infty)\wedge \widetilde{L}(1,\infty)=\widehat{L}(1,\infty)\right)=0
$$
by Lemma~\ref{nomatch}. By the identical argument, we also get that $\P\left(L(-2,\infty)=\infty\}\cap C_0\right)=0$. Note that our argument does not depend on  the choice of $x$. Hence, \eqref{5points} is proved. Combining \eqref{3points} with \eqref{5points}, it follows that
$$\P\left( \exists x\in \Z: L(x,\infty)=\infty, L(x\pm 1)<\infty, L(x\pm2)<\infty  \right)=1.$$
Finally, let us fix $x\in \Z$ and assume that $L(x,\infty)=\infty, L(x\pm 1)<\infty, L(x\pm2)<\infty$. Notice that the jumps from $x$ to $x+2$ occur no more often that the events of the Poisson process with rate $w\left(L(x+2,\infty)\right)<\infty$. Since $\mathbf{X}$, however, spends only a finite amount of time staying at $x+1$, only finitely many such jumps will occur. The same holds for the jumps from ${x-1}$ to ${x-2}$. Hence $\mathbf{X}$ will eventually get stuck to $\{ x-1,x,x+1\}$.
\end{proof}

\section{Appendix}
\begin{proof}[Proof of Lemma \ref{martingale}]
We first show that $\mathbf M$ is a martingale. Denote $I_t:=\1_{\{\widetilde X_t=1\}}$. For small $h>0$, we have
\begin{align}\label{eq:differential}
\E[I_{t+h}-I_t|\mathcal{F}_t]
\nonumber &= \sum_{j\in\{0,1\}} (\1_{j=1}-I_t) \P[\widetilde X_{t+h}=j|\mathcal{F}_t]\\
\nonumber &=- I_t w\left(\widetilde{L}(0,t)\right)\cdot h+(1-I_t)w\left(\widetilde{L}(1,t)\right)\cdot h+o(h)\\
 &= \left(\1_{\{\widetilde X_t=0\}} W(1,t)-\1_{\{\widetilde X_t=1\}} W(0,t)\right)\cdot h + o(h).
\end{align}
Let $0<s<t$ and $n$ be fixed. Define $t_j=s+j(t-s)/n$ with $j\in\{0,1,\dots,n\}$. By virtue of \eqref{eq:differential} and the law of iterated expectations, we have   
\begin{align*}   \E\left[ I_t-I_s|\ \mathcal{F}_s  \right] & =\E\left[\left. \sum_{j=1}^n \E\left[I_{t_{j}}-I_{t_{j-1}}\ |\ \mathcal{F}_{t_{j-1}} \right]\ \right|\ \mathcal{F}_s \right]\\
& =\E\left[ \left.\sum_{j=1}^n \Pi(t_{j-1})(t_j-t_{j-1}) +n\cdot o\left(\frac{t-s}{n}\right) \ \right| \ \mathcal{F}_s  \right],
\end{align*}
where we set $\Pi(t)=\1_{\{\widetilde X_t=0\}} W(1,t)-\1_{\{\widetilde X_t=1\}} W(0,t)$.
Notice that the left hand side does not depend on $n$ while we can take the limit of the right hand side as $n\to\infty$ thanks to Lebesgue's dominated convergence theorem. Hence, we obtain that
$$\E\left[ I_t-I_s|\ \mathcal{F}_s  \right]=\E\left[\int_s^t \Pi(u)\rmd u\ | \ \mathcal{F}_s \right]$$ and thus $\E[ M_t|\ \mathcal{F}_s ]=M_s.$ We next turn to proving \eqref{eq:anglebracket}. Notice that
\begin{align*}
M_t^2
=&\; I_t- 2\left(M_t+\int_0^t \Pi(s)\rmd s\right)\int_0^t \Pi(s)\rmd s + \left(\int_0^t \Pi(s)\rmd s \right)^2\\
=&\; M_t+\int_0^t \Pi(s)\rmd s - 2M_t\int_0^t \Pi(s)\rmd s - \left(\int_0^t \Pi(s)\rmd s \right)^2\\
\stackrel{\eqref{changvar}}{=}&\; M_t + \int_0^t\Pi(s) ds - 2\left( \int_0^t M_s \Pi(s)\rmd s +\int_0^t\left(\int_0^s \Pi(u)\rmd u \right)\rmd M_s \right)  \\
& \quad - 2\int_0^t \Pi(s)\left(\int_0^s \Pi(u)\rmd u\right)ds \\
=&\; Q_t + \int_0^t \Pi(s) \rmd s - 2\int_0^t \Pi(s)I_s\rmd s= Q_t +\int_0^t\left(\1_{\{\widetilde X_s=0\}}W(1,s)+\1_{\{\widetilde X_s=1\}}W(0,s)\right) \rmd s,
\end{align*}
where we set $$Q_t:=M_t-2\int_0^t\left(\int_0^s \Pi(u)\rmd u \right)\rmd M_s,$$
which is a local martingale. The lemma is proved.
\end{proof}
\begin{proof}[Proof of Lemma \ref{boundedjump}]
We assume w.l.o.g. that $N_0=0$. We fix $0<n<\infty$ and set $\tau_n=\inf\{t: N_t\le - n\}$. Note that $N_{t\wedge\tau_n}$ is a martingale and $N_{t\wedge\tau_n}\ge N_{t\wedge\tau_n-}-C\ge -n-C$.   Therefore, $N_{t\wedge\tau_n}+n+C$ is a nonnegative martingale which converges almost surely to a finite random variable. Hence,
$$\left\{ \liminf_{t\to\infty}N_t>-\infty\right\} =\bigcup_{n=1}^{\infty}\left\{\tau_n=\infty\right\}\subseteq \left\{\lim_{t\to\infty}N_t \ \text{exists and is finite}\right\}.$$
Similarly, replacing $N_t$ by $-N_t$, we also have
$$\left\{ \limsup_{t\to\infty}N_t<+\infty\right\} \subseteq \left\{\lim_{t\to\infty}N_t \ \text{exists and is finite}\right\}.$$
\end{proof}
\section*{Acknowledgement} {A.C. and T.M.N.  work is partially supported by ARC grant  DP180100613. A.C. is also supported by ARC Centre of Excellence for Mathematical and Statistical Frontiers (ACEMS) CE140100049. S.V.  research is partially supported by Crafoord grant no. 20190667 and Swedish Research Council grant VR 2019-04173. The authors would like to thank Masato Takei and two anonymous referees for their thorough reading and their constructive suggestions which improved the manuscript.}

\end{document}